\documentclass[12pt]{scrartcl} 
\usepackage{array}
\usepackage{supertabular}
\usepackage{amsmath,amsfonts,amssymb,amsthm} 
\usepackage{bm}
\usepackage{graphicx}
\usepackage{subcaption}

\newtheorem{theo}{Theorem}

\newcommand{\xf}{\bm{X}^{(1)}}
\newcommand{\xs}{\bm{X}^{(2)}}

\newcommand{\mso}[1]{\xs_{(#1)}}
\newcommand{\mf}[1]{m_{#1}^{(1)}}
\newcommand{\ms}[1]{m_{#1}^{(2)}}
\newcommand{\mfe}[1]{\hat{m}_{#1}^{(1)}}
\newcommand{\mse}[1]{\hat{m}_{#1}^{(2)}}
\newcommand{\bj}{\bm{j}}

\title {Criterion for the resemblance between the mother and the model distribution}
\author{Yo Sheena\thanks{Faculty of Data Science, Shiga University, Japan; Visiting Professor of the Institute of Statistical Mathematics, Japan}}

\date{\today}

\begin{document}
\maketitle
	
\begin{abstract}
 If the probability distribution model aims to approximate the hidden mother distribution, it is imperative to establish a useful criterion for the resemblance between the mother and the model distributions. 

This study proposes a criterion that measures the Hellinger distance between discretized (quantized) samples from both distributions. Unlike information criteria such as AIC, this criterion does not require the probability density function of the model distribution, which cannot be explicitly obtained for a complicated model such as a deep learning machine. Second, it can draw a positive conclusion (i.e., both distributions are sufficiently close) under a given threshold, whereas a statistical hypothesis test, such as the Kolmogorov-Smirnov test, cannot genuinely lead to a positive conclusion when the hypothesis is accepted.

In this study, we establish a reasonable threshold for the criterion deduced from the Bayes error rate and also present the asymptotic bias of the estimator of the criterion. From these results, a reasonable and easy-to-use criterion is established that can be directly calculated from the two sets of samples from both distributions.
\end{abstract}
\noindent
MSC(2010) \textit{Subject Classification}: Primary 60F99, Secondary 62F12

\section{Introduction}
\ \ Consider a probability distribution model (the term ``probabilistic generative model'' is also used in some contexts), which generates a random vector. Although the purpose of making a model is quite diverse, here we focus on approximating a pre-existing distribution using the model. We call the preexisting distribution - ``mother distribution'' and the distribution of the model - ``model distribution.'' In typical cases, ``real'' data obtained via observation or an experiment is supposed to be generated from the unknown mother distribution, while the data from the model is a  ``mimic'' but is expected to resemble the real one.

Once we obtain a distribution model, the resemblance between the model and the mother distributions must be evaluated. The most commonly used criterion for the closeness between two distributions is divergence. In particular, the Kullback-Leibler divergence
\begin{equation}
\label{KL_dive}
\int_{\mathcal{X}} g_r(x) \log\frac{g_r(x)}{g_m(x)} d\mu(x) = \int_{\mathcal{X}} g_r(x) \log g_r(x) d\mu(x) - \int_{\mathcal{X}} g_r(x) \log g_m(x) d\mu(x)
\end{equation}
is the most frequently used standard, where $g_r(x)$ and $g_m(x)$ are the probability density functions (p.d.f.) of the mother (i.e., real) and model distributions with respect to the common measure $\mu$ over domain $\mathcal{X}$, respectively. In most cases, $g_r(x)$ is unknown and can be estimated by using the sample from the mother distribution. Nonparametric density estimation is a difficult task that requires a large sample size, especially in a high-dimensional distribution, and is not robust because its performance often depends on the derivatives of the density (e.g., Theorem 6.1 of \cite{Scott}). A popular alternative is the average log-likelihood (e.g., \cite{Borji}, \cite{Theis_et_al}); Let $\bm{X}_\tau,\ \tau=1,\ldots, n$ be a sample of size $n$ from the mother distribution.
Subsequently, the average log-likelihood
\[
\frac{1}{n} \sum_{\tau=1}^n \log g_m(\bm{X}_\tau)
\]
is the estimator of minus the second part of the right-hand side of \eqref{KL_dive} (i.e., the cross-entropy between $g_r(x)$ and $g_m(x)$). Since the first part is common among the models, the larger log-likelihood of one model than another indicates that the model is closer to the mother distribution. AIC is given as the bias correction of the average log-likelihood when the parametric model is chosen with the maximum likelihood estimation. However, the average log-likelihood (or information criterion) has the following shortcomings:
\begin{enumerate}
\item A concrete form of $g_m(x)$ is needed, which is difficult to obtain for a model as complex as Bayes model and  multi-layer neural networks (``deep generating models''; refer to \cite{Ruthotto&Haber} for more details) 
\item Though it works for model comparison, the value itself has no statistical meaning; hence, it cannot be used for model evaluation. For example, it gives no clear answer to questions such as whether training is done well, whether the parametric model is appropriate, or whether the sample size is large enough.
\end{enumerate}

A possible remedy for the first shortcoming is evaluating the model's closeness to the mother distribution ``directly'' (not through density estimation)from the two samples:
\begin{equation}
\label{two_sample}
\begin{aligned}
&\text{Observed data of the mother distribution: }\xf \triangleq \{\xf_1,\ldots \xf_{n_1}\},\\
&\text{Generated sample from the model: }\xs \triangleq \{\xs_1,\ldots,\xs_{n_2}\},
\end{aligned}
\end{equation}
where $n_1$ and $n_2$ are the sample sizes. The closeness between the empirical distributions of $\xf$ and $\xs$ guarantees closeness between the mother and model distributions for a sufficiently large sample size. Furthermore, if a reasonable (persuasive) criterion for the closeness between the two empirical distributions is set, the second shortcoming can be overcome. Therefore, we can conclude that the mother and model distributions are close to each other if $\xf$ and $\xs$ satisfy certain closeness criteria.


Statistical testing is the most frequently used method in this approach, and nonparametric tests, such as the Kolmogorov-Smirnov (K-S) test, are often used to test the null hypothesis that the two samples are generated from the same distribution. However, statistical testing has its drawbacks. If the null hypothesis is rejected, we can conclude that the two population distributions are not the same, but they still might be quite similar for practical purposes. When the sample sizes are large, the null hypothesis is likely to be rejected because the test detects a slight difference between the samples. Conversely, clear conclusions on closeness cannot be reached if the hypothesis is accepted. Besides acceptance or rejection, the p-value provides some information, but it does not directly address the closeness between the two distributions. 

In this study, a simple criterion is proposed to judge the closeness between the mother and model distributions based on samples $\xf$ and $\xs$. Using these criteria, we can reach the ``positive'' conclusion that the mother and model distributions are sufficiently close. It is often the case that we need the ``positive'' conclusion for closeness rather than the ``negative'' conclusion that the hypothesis rejection draws. To compare the two samples $\xf$ and $\xs$, we adopt discretization (the term ``quantification'' is also used). Discretization is a naive but versatile parameterization that is applicable to any continuous distribution.

Our approach is as follows:
\begin{enumerate}
\item Make the partition of $\mathcal{X}$ based on the value of $\xs$ (each partition is referred to as ``bin'').
\item Make the multinomial empirical distributions by counting the number of individuals that fall into each bin for both $\xf$ and $\xs$.
\item Evaluate the closeness between the two multinomial empirical distributions with the Hellinger distance.
\item Compare the value of the Hellinger distance with a certain threshold and conclude whether the two distributions are sufficiently close or not.
\end{enumerate}
As explained later, the first step is the ``moving region method,''  where the bins are chosen based on sample values, while the ordinary ``fixed region method'' makes bins independent of the samples. 
(refer to \cite{Sheena_4}.)  Note that $\xs$ was used for this method because it is usually much more costly to observe a large amount of data from the mother distribution than to generate a large sample size from the model, which usually incurs only a computational cost (refer to Theorem \ref{bias_two_sample}). 

We briefly introduce relevant works for the comparison between two distributions (not necessarily the mother and model distributions). Richardson and Weiss \cite{Richardson&Weiss} take the most similar approach to the present paper, but they adopt the ``fixed region method.'' They considered the collection of binomial distributions (over one bin and the others) rather than the multinomial distribution and used the z-test for the final evaluation of closeness. Johnson and Dasu \cite{Johnson&Dasu} proposed a comparison based on partitioning high-dimensional data into categorical and continuous variables, using multiple tests to determine the similarity between the two datasets. For an extension of the K-S test for a high-dimensional case, refer to Press and Teukolsky \cite{W. H. Press and S. A. Teukolsky} (two-dimensional), Hagen et al. \cite{Hagen_et.al} (general dimension), and the papers therein. Loudin and Miettinen \cite{Loudin and Miettinen} proposed a method to condense the information of two datasets into a one-dimensional distribution and used a one-dimensional K-S test for the final decision.

Gretton et al. \cite{Gretton_et.al} applied the theory of reproducing kernel Hilbert space to a two-sample test. Their approach is novel as the similarity between two distributions is measured by the similarity between the two corresponding functions. However, the authors \cite{Gretton_et.al} choose ordinary statistical testing for the final evaluation. Theis et al. \cite{Theis_et_al} evaluated probabilistic (image) generating models, showing that some evaluation methods lead to completely different conclusions. Borji \cite{Borji} reviewed the papers thoroughly related to the evaluation measures of generative adversarial networks, some of which are applicable to the two-sample problems of this study.

This research is organized as follows. In the next section, we introduce $f$-divergence and present the relation between $f$-divergence and the Bayes error rate. Using this relation, we can interpret the value of $f$-divergence from a more practical perspective. In Section \ref{dis_methods}, two discretization methods are introduced: the fixed and moving region methods for the case when $k=1$ (Section \ref{Discretization of Scalar}) and the general $k$ (Section \ref{discre_vector_X}). In Section \ref{One Sample Problem}, we state the asymptotic risk for `` the one-sample problem,'' (i.e., the closeness between the true distribution and the predictive distribution), which is closely related to the two-sample problem. In Section \ref{Two Sample Problem}, we propose an estimator of the closeness between the mother and model distributions. Thereafter, the asymptotic bias of the estimator is evaluated. Finally, in Section \ref{criteria}, a simple criterion for the closeness between the two distributions is proposed by combining all the results. We also present the simulation results to demonstrate how these criteria work.
%
%
%
%
%
%
\section{$f$-Divergence}
\label{dive}
The discrepancy between the two probability distributions can be measured by divergence. Consider the two distributions over $\mathcal{X}$ whose probability density functions with respect to a common measure $d\mu$ are given by $g_1(x)$ and $g_2(x)$, respectively. The $f$-divergence between these distributions is given by:
\begin{equation}
\label{f-divergence}
D_f[g_1(x) \,| \, g_2(x)] \triangleq \int_{\mathcal{X}} g_1(x) f\biggl(\frac{g_2(x)}{g_1(x)}\biggr) d\mu(x),
\end{equation}
where $f$ is a convex function with properties $f(1)=f'(1)=0$ and $f''(1)=1$. (Refer to Csiszar \cite{Csiszar} and Amari \cite{Amari4} for more details on $f$-divergence.)
Specifically, for the two multinomial distributions whose probabilities for each bin are given by $m^{(1)}=(m^{(1)}_0,\ldots,m^{(1)}_p)$ and $m^{(2)}=(m^{(2)}_0,\ldots,m^{(2)}_p)$, respectively, $f$-divergence is given by
\begin{equation}
\label{def_fdive}
D_f[m^{(1)} : m^{(2)}] \triangleq
\sum_{i=0}^p m^{(1)}_i \:f\biggl(\frac{m^{(2)}_i}{m^{(1)}_i}\biggr).
\end{equation}

$f$-divergence $D_f[m^{(1)}:m^{(2)}]$ satisfies the following condition:
\begin{equation}
D_f[m^{(2)} : m^{(1)}]\geq 0,\qquad D_f[m^{(2)} : m^{(1)}]=0 \text{ if and only if $m^{(1)} = m^{(2)}$}.
\end{equation}
However, triangle inequality and symmetricity do not hold true. To satisfy the symmetricity, the symmetrized divergence 
\begin{equation}
\frac{1}{2}\left\{
D_f[m^{(1)} : m^{(2)}]+D_{f^*}[m^{(1)} : m^{(2)}]
\right\}
\end{equation}
was also used (refer to Amari and Cichocki \cite{Amari&Cichocki}), where $f^*(x) \triangleq xf(1/x)$ is the dual function of $f(x)$.
Symmetricity arises from the fact that 
\begin{equation}
\label{dual_f}
D_{f^*}[m^{(1)} : m^{(2)}] = D_f[m^{(2)} : m^{(1)}]
\end{equation}

For a more concrete analysis of the closeness between the two distributions, $\alpha$-divergence is commonly used. $\alpha$-divergence is a one-parameter ($\alpha$) subfamily of $f$-divergence given by \eqref{f-divergence} with $f(x)=f_\alpha(x)$ such that
\begin{equation}
\label{def_alphadive}
f_\alpha(x)\triangleq 
\begin{cases}
\frac{4}{1-\alpha^2}\bigl(1-x^{(1+\alpha)/2}\bigr)+\frac{2}{1-\alpha}(x-1) &\text{ if $\alpha \ne \pm 1$,} \\
x\log x +1-x & \text{ if $\alpha=1$,}\\
-\log x+x-1 & \text{ if $\alpha=-1$.}
\end{cases}
\end{equation}
$\alpha$-divergence is still a broad class that contains frequently used divergences, such as the Kullback-Leibler divergence ($\alpha =-1$), the Hellinger distance ($\alpha=0$), and $\chi^2$-divergence ($\alpha=3$). We will use the notation $\overset{\alpha}{D}[m^{(1)} : m^{(2)}]$ instead of $D_{f_\alpha}[m^{(1)} : m^{(2)}]$. The symmetrized $\alpha$-divergence is denoted by $\overset{|\alpha|}{D}[m^{(1)} : m^{(2)}]$, which is
\[
\overset{|\alpha|}{D}[m^{(1)} : m^{(2)}] = \frac{1}{2}\left\{\overset{\alpha}{D}[m^{(1)} : m^{(2)}]+\overset{-\alpha}{D}[m^{(1)} : m^{(2)}]\right\}
\]
because $D_{f^*_\alpha} = D_{f_{-\alpha}} =\overset{-\alpha}{D}$,

The value of $f$-divergence is difficult to interpret; the value itself of $D_f[ m^{(1)} : m^{(2)}]$ does not indicate how close $m^{(1)}$ and $m^{(2)}$ are. Given a specific value of $f$-divergence, we want to judge whether it is sufficiently small. Afterward, the relation between $f$-divergence and the Bayes error rate is established. This relation enables us to provide a reasonable (plausible) interpretation of the value of $f$-divergence.

If any two distributions (represented by their p.d.f's $g_1(x)$ and $g_2(x)$) are sufficiently close, it is difficult to discriminate samples from both distributions. The Bayes discriminant rule classifies sample $X$ into a distribution with a larger $g_i(X)$. This rule minimizes the error rate (probability), and the minimum rate (called the Bayes error rate) is
\[
Er[g_1(x) | g_2(x) ] \triangleq \frac{1}{2}\int_{\mathcal{X}} \min(g_1(x), g_2(x))d\mu.
\]
The following theorem reveals the relation between $f$-divergence and the Bayes error rate.

Let $\Delta_*(\delta)$ be the solution of the equation with respect to $\Delta(\geq 1)$
\begin{equation}
\label{sol_Delta}
\Bigl(\frac{1}{\Delta}\Bigr)f(\Delta) +
\Bigl(1-\frac{1}{\Delta}\Bigr)f(0)=\delta (>0)
\end{equation}
and
\begin{equation}
\label{def_A}
A(\delta)\triangleq \Bigl\{ (x,t) \,\Big|\, x f\Bigl(\frac{1-2t}{x} + 1\Bigr)+(1-x) f\Bigl(\frac{2t-1}{1-x}+1\Bigr) = \delta,\quad 0< x < 2t < 1\Bigr\}.
\end{equation}
%
%
%
%
\begin{theo}
\label{Theo:Error and Div}
If $D_f[g_1(x) \,| \, g_2(x)] < \delta$, then
\begin{equation}
\label{rel_error_D}
Er[g_1(x)\, | \, g_2(x)] \geq \alpha(\delta),
\end{equation}
where 
\[
\alpha(\delta) = \min\bigl(1/(2\Delta_*(\delta)), \inf\{ t \,|\, (x,t) \in A(\delta) \}\bigr).
\]
\end{theo} 
\begin{proof}
Refer to Appendix. 
\end{proof}
Using the result of the theorem, a reasonable threshold $\delta^*$ can be set for $f$-divergence. First, we choose $\epsilon$ (small values such as 0.05, 0.01) to construct a standard
\[
\frac{1}{2} - \epsilon
\]
for the error rate. This error rate is close to 0.5, i.e. that of the ``random guess''. If the Bayes error rate  equals $0.5-\epsilon$, we can conclude that the two distributions are quite similar.  Let $\delta^*$ denote the solution to the equation
\begin{equation}
\label{eq_del_al}
\alpha(\delta) = \frac{1}{2} - \epsilon
\end{equation}
If the $f$-divergence is smaller than $\delta^*$, the two distributions are so close that they cannot be distinguished well even when the p.d.f.s of the two distributions are known.

In Section \ref{criteria}, we establish the criteria for closeness between the mother and model distributions using the Hellinger distance. In preparation for the section, $\alpha(\delta)$ in Theorem \ref{Theo:Error and Div} is investigated for the Hellinger distance, which is $f$-divergence with
\[
f(x) = 2(1-\sqrt{x})^2.
\]
As $f(0)=2$, \eqref{sol_Delta} is equivalent to
\[
\frac{2}{\Delta}(1-\sqrt{\Delta})^2 + 2\Bigl(1-\frac{1}{\Delta}\Bigr) = \delta,
\]
and the solution is given by:
\begin{equation}
\label{Delta*}
\Delta_*(\delta) = \frac{1}{(1-\delta/4)^2}.
\end{equation}
As $\inf\{ t \,|\, (x,t) \in A(\delta) \}$ is complicated, we approximate it  around $t=1/2$. Through Taylor expansion,
\begin{align*}
f\Bigl(\frac{1-2t}{x}+1\Bigr) &= f(1) + \frac{1-2t}{x}f'(1) + \Bigl(\frac{1-2t}{x}\Bigr)^2\frac{f''(1)}{2} + \cdots\\
&\doteqdot \frac{(1-2t)^2}{2x^2}.
\end{align*}
Similarly, we have
\[
f\Bigl(\frac{2t-1}{1-x}+1\Bigr) \doteqdot \frac{(1-2t)^2}{2(1-x)^2}.
\]
Therefore, we can approximate $A(\delta)$ as
\begin{align*}
A(\delta) &\doteqdot \Bigl\{(x,t)\ \Big| \  (1-2t)^2 \Bigl(\frac{1}{2x} + \frac{1}{2(1-x)}\Bigr) = \delta, \quad 0 < x < 2t <1 \Bigr\}\\
&= \Bigl\{(x,t) \  \Big|\  (1-2t)^2 = 2\delta x(1-x), \quad 0 < x < 2t < 1\Bigr\} \\
&= \Bigl\{(x,t) \  \Big|\  t = \Bigl(1-\Bigl(2\delta x(1-x)\Bigr)^{1/2}\Bigr)/2, \quad 0 < x < 2t < 1\Bigr\}.
\end{align*}
Using this approximation of $A(\delta)$, if $\delta \leq 1/2$, then  $\inf\{t \ | \ (x,t) \in A(\delta) \} = \Bigl(1-\sqrt{\delta/2}\Bigr)/2$. (Refer to Figure 1.)
\begin{figure}
\centering
\includegraphics[width=8cm]{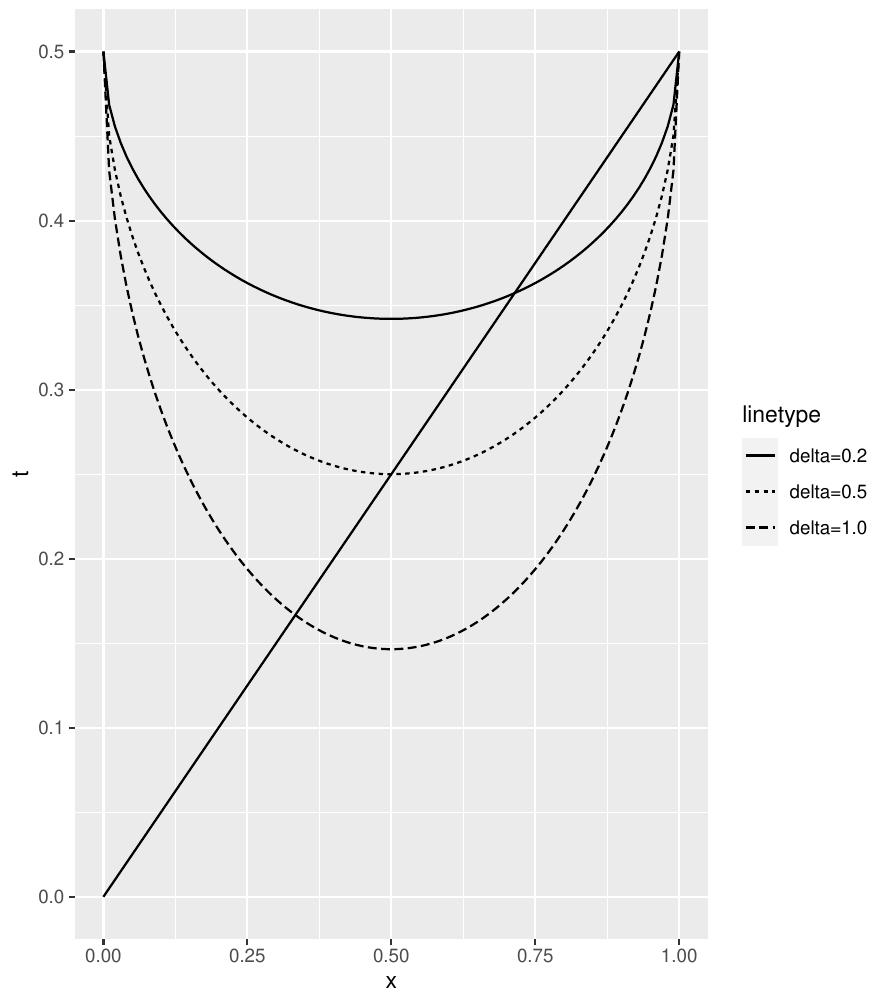}
\caption{Graph of $t(x) = \Bigl(1-\Bigl(2\delta x(1-x)\Bigr)^{1/2}\Bigr)/2$ with $\delta=0.2,0.5,1.0$}
\end{figure}
Note that 
\[
(1-\delta/4)^2/2 > \Bigl(1-\sqrt{\delta/2}\Bigr)/2.
\]
Consequently, $\alpha(\delta)$ is given approximately by $\Bigl(1-\sqrt{\delta/2}\Bigr)/2$.
The solution \eqref{eq_del_al} with this approximated $\alpha(\delta)$ is equal to:
\begin{equation}
\label{simple_threshold}
\delta^* = 8\epsilon^2.
\end{equation}
The threshold $\delta^*$  is  $1/50(1/1250)$ when $\epsilon=0.05(0.01)$. The same approximation holds for the Kullback-Leibler divergence (refer to Corollary 1 in Sheena \cite{Sheena_5}). 
%
%
%
%
%
%
\section{Discretization Methods}
\label{dis_methods}
In this section, we consider two contrasting discretization methods: the fixed-region method and the moving-region method. Consider 
\[
\bm{X}=(X_1,\ldots,X_k)
\]
as a $k$-dimensional random vector. Its support is given by
\[
\mathcal{X}=[a_1 \ b_1]\times [a_k\ b_k],
\]
where it is possible for $a_i=-\infty$ and $b_i=\infty$.

First, these methods are discussed for the case where $\bm{X}$ is one-dimensional (i.e., $k=1$). In this case, the terms ``fixed interval method'' and ``moving interval method'' are used for both methods. Next, we generalize the discretization to the general case when $k\geq 2$.
%
%
%
%
\subsection{Discretization of Scalar $X$}
\label{Discretization of Scalar}
For the case $k=1$, we use notation $X$ instead of $\bm{X}$. The fixed interval method is an ordinary discretization method, and the intervals are fixed independently from the sample values. In other words, the intervals $I_i,\ i=1,\ldots, p$ such that
\[
\mathcal{X}=\bigcup_{i=1}^{p} I_i,\qquad I_i \cap I_j = \phi\ (i\ne j).
\]
were administered before taking the sample.

The second method is the moving interval method.
We fix points $\lambda_i (1\leq i \leq p)$ in the interval $(0, 1)$.
\begin{equation}
\label{def_per}
\lambda_0(\triangleq 0)< \lambda_1 < \lambda_2 < \cdots < \lambda_p <\lambda_{p+1}(\triangleq 1).
\end{equation}
and 
\begin{equation}
\label{def_per_mother}
\xi_i \triangleq F^{-1}(\lambda_i),\ 1\leq i \leq p, \quad \xi_0\equiv a_1,\quad \xi_{p+1}\equiv b_1,
\end{equation}
where $F^{-1}$ is the inverse function of the cumulative distribution function $F$ of the mother distribution. Namely,  $\xi_i$ is the percentile of the mother distribution corresponding to $\lambda_i$.

Let $\hat{\xi}_i$ denote the estimator of $\xi_i$ for $i=1,\ldots p$ ($\hat{\xi}_0\equiv a_1$ and $\hat{\xi}_{p+1}\equiv b_1$). Subsequently, the moving intervals
\[
I_i \triangleq (\hat{\xi}_i,\ \hat{\xi}_{i+1}) \quad 0\leq i \leq p,
\]
are defined. 

Let $X_\tau,\ \tau=1,\ldots,n$ be an i.i.d. sample of size $n$ following the distribution of $X$. If we simply estimate percentile $\xi_i$ using the order statistic, 
\[
X_{(1)} \leq X_{(2)} \leq \cdots \leq X_{(n)}, 
\]
Then, $\hat{\xi} _i $ is given by: 
\begin{equation}
\hat{\xi}_i\triangleq X_{(\tilde{n}_i)}\quad 1\leq i \leq p,
\end{equation}
where $\tilde{n}_i = \lfloor n\lambda_i \rfloor$ ($\lfloor \cdot \rfloor$ denotes the floor function). If we choose 
\begin{equation}
\label{equal_mass}
\lambda_i=i/(p+1),\ i=1,\ldots,p,
\end{equation}
each interval contains (almost) the same number of individuals (``equal mass partitions '').

Currently, there is no standard or systematic method for deciding the bins for fixed interval methods. Contrariwise, the moving interval method has a natural choice of $\lambda_i$ \eqref{equal_mass}. In practical situations, a mixture of both methods is often applied. For instance, when we draw a histogram, we usually determine the bin number or endpoints of the intervals after observing the sample values.

%
%
%
%
%
%
\subsection{Discretization of Random Vector $\bm{X}$}
\label{discre_vector_X}
We generalize the discretization for the scalar ($k=1$) to that for the general $k$-dimensional random vector, $\bm{X}$. 
For the fixed region, which is independent of the sample, the generalization can be simply made by substituting the partition by intervals with that of mutually exclusive regions $I_i,\ i=0,\ldots, p'$ of any shape. 
\begin{equation}
\label{partition_fix}
\mathcal{X}=\bigcup_{i=0}^{p'} I_i.
\end{equation}

Next, we consider the moving region method for the case $k\geq 2$, namely, discretization  based on the sample. Let an i.i.d. sample of size $n$ from the distribution of $\bm{X}$ be denoted by:
\[
\bm{X}_\tau = (X_{1\tau},\ldots,X_{k\tau}),\qquad \tau=1,\ldots,n.
\]
Although there are many possible partitions of $\mathcal{X}$ based on the sample, we consider a simple recursive method using the order statistics along each coordinate in this study. 

The partition was performed according to the following steps: \bigskip
\\
-- First Step --\\
We partition the entire sample into $p+1$ regions using the following order statistics: 
\[
X_1(1) \leq \cdots \leq X_1(n),
\]
derived from $X_{11},\ldots,X_{1n}$. We define the integer $\tilde{n}_{j}$ 
and interval $I_{j}$ as
\begin{equation}
\label{def_I_i}
\begin{split}
&\tilde{n}_{j} \triangleq \lfloor(n j )/(p+1) \rfloor, \qquad  j=0,\ldots,p+1,\\
&I_{j} \triangleq  (X_1(\tilde{n}_{j} ),\ X_1(\tilde{n}_{j+1})], \quad X_1(\tilde{n}_{0}) \triangleq a_1,\ X_1(\tilde{n}_{p+1})\triangleq b_1, \qquad j=0,\ldots,p. 
\end{split}
\end{equation}
Each region $A_{j} (j=0,\ldots,p)$ is defined as follows:
\[
A_{j}\triangleq\bigl\{  \bm{X} \in \mathcal{X} \,\big| \,X_{1} \in I_{j} \bigr\}
\]
Let $n(A_j)$ denote the number of individuals belonging to $A_j$, which is equal to $\tilde{n}_{j+1}-\tilde{n}_{j}$.
\\
\\
-- Second Step -- \\
We partition each region $A_{j} (j=0,\ldots,p)$ into $p_j+1$ regions using the following order statistics: 
\[
X_{j2}(1) \leq \cdots \leq X_{j2}(n(A_j))
\]
from the individuals of the sample that belong to $A_j$ in ascending order in the second coordinate. Consider
\[
\tilde{n}_{j,k} \triangleq \lfloor n(A_j) k /(p_j+1) \rfloor,\quad k=0,\ldots,p_j+1
\]
and for  $k=0,\ldots,p_j$, 
\[
I_{jk} \triangleq  (X_{j2}(\tilde{n}_{j,k} ),\ X_{j2}(\tilde{n}_{j,k+1})], 
\]
where
\[
X_{j2}(\tilde{n}_{j,0}) \triangleq a_2,\ X_{j2}(\tilde{n}_{j,p_j+1})\triangleq b_2.
\]
We define the new region $A_{jk} (k=0,\ldots,p_j)$ as
\[
A_{jk}\triangleq\bigl\{ \bm{X} \in A_j \big|  X_{2} \in I_{jk} \bigr\}.
\]
The number of individuals that belong to this region is denoted by $n(A_{jk})$.

We proceed similarly by partitioning each region made in the previous step by using order statistics with respect to another variate. In general, the $i$th step ($i=2,\ldots,k$) is given as follows:
\\
\\
-- $i$th Step --\\
We partition region $A_{j_1 j_2 \ldots j_{i-1}}$ into $p_{j_1 j_2 \ldots j_{i-1}}+1$ regions using order statistics.
\[
X_{j_1 j_2 \ldots j_{i-1} i}(1) \leq \cdots \leq X_{j_1 j_2 \ldots j_{i-1} i}(n(A_{j_1 j_2 \ldots j_{i-1}}))
\]
made from the individuals of the sample belonging to $A_{j_1 j_2 \ldots j_{i-1}}$ in ascending order in the $i$th coordinate. Let
\[
\tilde{n}_{j_1,\ldots,j_{i-1},j_i} \triangleq \lfloor n(A_{j_1 j_2 \ldots j_{i-1}}) j_i/(p_{j_1 j_2 \ldots j_{i-1}}+1) \rfloor,\quad j_i=0,\ldots,p_{j_1 j_2 \ldots j_{i-1}}+1
\]
and 
\[
I_{j_1 j_2 \ldots j_{i-1} j_i}  \triangleq  (X_{j_1 j_2 \ldots j_{i-1} i}(\tilde{n}_{j_1,\ldots,j_{i-1},j_i}),\ X_{j_1 j_2 \ldots j_{i-1} i}(\tilde{n}_{j_1,\ldots,j_{i-1},j_i+1})], \quad j_i=0,\ldots,p_{j_1 j_2 \ldots j_{i-1}}
\]
where
\[
X_{j_1 j_2 \ldots j_{i-1} i}(\tilde{n}_{j_1,\ldots,j_{i-1},0}) \triangleq a_i, \quad X_{j_1 j_2 \ldots j_{i-1} i}(\tilde{n}_{j_1,\ldots,j_{i-1},p_{j_1 j_2 \ldots j_{i-1}}+1})\triangleq b_i.
\]
The new region $A_{j_1 j_2 \ldots j_{i-1} j_i} (j_i=0,\ldots,p_{j_1 j_2 \ldots j_{i-1}})$ is defined as:
\[
A_{j_1 j_2 \ldots j_{i-1}  j_i}\triangleq\bigl\{ \bm{X} \in  A_{j_1 j_2 \ldots j_{i-1}} \big| X_{i} \in I_{j_1 j_2 \ldots j_{i-1} j_i} \bigr\},
\]
where size is denoted by $n(A_{j_1 j_2 \ldots j_{i-1}  j_i})$. 

When we finish the $l$th step ($l \leq k$), the partition is completed (we call $l$ ``the partition depth'');
\begin{equation}
\label{partition_moving}
\mathcal{X} = \bigcup_{\bj(l) } A_{\bj(l)},\qquad \bj(l) \triangleq (j_1,\ldots,j_l),
\end{equation}
where each $A_{\bj(l)}$ is mutually exclusive. 
Note that this coordinate-wise recursive partition could include a broad class partition if we use a coordinate transformation (e.g., a linear transformation, polar coordinate system, etc.).

For readability, we adopted the following notations:
\[
\bj(1) \triangleq (j_1),\quad \bj(2) \triangleq (j_1, j_2),\quad \cdots \quad\bj(l) \triangleq (j_1,\ldots,j_l),
\]
and for $ i=1,\ldots,l, \ s=0,\ldots,p_{j_1 j_2 \ldots j_{i-1}}+1,$
\begin{align*}
&A_{\bj(i)} \triangleq A_{j_1\ldots j_i},\quad I_{\bj(i)} \triangleq I_{j_1 \ldots j_i},\quad p_{\bj(i)} \triangleq p_{j_1\ldots j_{i}},\\
& X_{\bj(i-1)i} (\tilde{n}_{\bj(i-1)s}) = X_{j_1 j_2 \ldots j_{i-1} i}(\tilde{n}_{j_1,\ldots,j_{i-1},s}).
\end{align*}
%
%
%
%
%
%
\section{One Sample Problem}
\label{One Sample Problem}
In this section, we consider the ``one sample problem,''  which deals with the closeness between the true distribution of $\bm{X}$ and the predicted distribution based on i.i.d. sample of size $n$ from the true distribution. The result of this problem will be used for the main issue of this paper, the ``two-sample problem,'' in the next section.

For the one-sample problem, the fixed region method and the moving region method are formulated as follows. For the fixed region method, suppose that the partition of $\mathcal{X}$ is given by \eqref{partition_fix}. Then, we have two multinomial distributions. The discretized true distribution of $\bm{X}$ is given by the parameter
\begin{equation}
\label{true_m_fixed}
m=(m_0,\ldots,m_{p'}),\qquad m_i = P(\bm{X} \in I_i),\quad i=0,\ldots,p',
\end{equation}
whereas the corresponding maximum likelihood estimator (MLE) 
\begin{equation}
\label{est_m_fixed}
\hat{m}=(\hat{m}_0,\ldots,\hat{m}_{p'}),
\end{equation} 
is given by the sample ratio of each region $I_i$. 

Suppose that the partition of $\mathcal{X}$ based on the i.i.d. sample is given by \eqref{partition_moving}. For the moving region method, the discretized true distribution of $\bm{X}$ is given by the parameter
\begin{equation}
\label{true_m_moving}
\hat{m} \triangleq \{\hat{m}_{\bj(l)}\}_{\bj(l)},\qquad \hat{m}_{\bj(l)}, \triangleq P(\bm{X} \in A_{\bj(l)}),
\end{equation}
whereas its estimator is given by
\begin{equation}
\label{est_m_moving}
m \triangleq \{m_{\bj(l)}\}_{\bj(l)},\qquad m_{\bj(l)} = 1/(p'_{\bj(l-1)} +1),
\end{equation}
where $p'_{\bj(l-1)} \triangleq (p+1)(p_{\bj(1)}+1)\cdots (p_{\bj(l-1)}+1)-1$ and the lower index $\bj(l)$ outside the braces implies that $\bj(l)$ runs through all possible cases.

We measure the discrepancy between the two multinomial distributions represented by the parameters $m$ and $\hat{m}$ through $f$-divergence. For the fixed region method,
\begin{equation}
\label{f-div_fix}
D_f[m : \hat{m}] \triangleq
\sum_{i=0}^{p'} m_i \:f\biggl(\frac{\hat{m}_i}{m_i}\biggr),
\end{equation}
And for the moving-region method,
\begin{equation}
\label{f-div_move}
D_f[m : \hat{m}] \triangleq
\sum_{\bj(l)} m_{\bj(l)} \:f\biggl(\frac{\hat{m}_{\bj(l)}}{m_{\bj(l)}}\biggr).
\end{equation}
For \eqref{f-div_move}, if one might think it is natural to consider $D_f[\hat{m} : m]$ in the sense that the true parameter should come first, it is satisfied using the dual function $f^*$ (refer to \eqref{dual_f}). Hence, we proceed with Eq. \eqref{f-div_move}.

The risks for the fixed region method and moving region method are denoted by $ED_I$ and $ED_P$, that is,
\begin{align*}
&ED_I \triangleq E\bigl[D_f[m : \hat{m}]\bigr],\qquad \text{$D_f[m : \hat{m}]$ is given by \eqref{f-div_fix}}\\
&ED_P \triangleq E\bigl[D_f[m : \hat{m}]\bigr],\qquad \text{$D_f[m : \hat{m}]$ is given by \eqref{f-div_move}}
\end{align*}

We briefly summarize the results of Sheena \cite{Sheena_4}, who studied the asymptotic efficiency of the fixed and moving interval methods in the scalar case. First, the moving interval method is asymptotically superior in the $n^{-2}$ order to any fixed interval method with respect to the risk of the symmetrized $\alpha$-divergence  when the $\lambda$ values are chosen as in \eqref{equal_mass} (see Theorem 3 of \cite{Sheena_4}). This superiority denotes that the fixed interval method requires a larger sample size to be on par with the moving interval method in terms of estimation efficiency. Second, the estimation efficiency of the fixed interval method relies heavily on the true distribution (Theorem 1 of \cite{Sheena_4}), whereas that of the moving interval method is independent of the true distribution (Theorem 2 of \cite{Sheena_4}). 

Now, we generalize the results of \cite{Sheena_4} for the scalar case into the vector case. Let $N$ denote the sample size.
%
%
%
%
\begin{theo}
\label{ED_I_one_sample}
Suppose that $f^{(5)}$ exists on $(0, \infty)$ and is bounded on $[\epsilon, \infty)$ for any $\epsilon (> 0)$. 
\begin{equation}
\label{ED_I_expan}
ED_I=\frac{p'}{2n}+\frac{1}{24n^2}\Bigl[ 4f^{(3)}(1)\Bigl(-3{p'}-1+M \Bigr)+3f^{(4)}(1)\Bigl(-2{p'}-1+M\Bigr)\Bigr]+o(n^{-2}),
\end{equation}
where
$$
M\triangleq \sum_{i=0}^{p'} m_i^{-1}.
$$
\end{theo}
\begin{proof}
This is obvious from the proof of Theorem 1 in \cite{Sheena_4}.
\end{proof}
%
%
%
%
%
\begin{theo}
\label{asymp_risk_moving}
Suppose that $f^{(3)}$ exists on $(0, \infty)$ and is bounded on $[\epsilon, \infty)$ for any $\epsilon (> 0)$. Then,
\begin{equation}
\label{ED_P_asymp}
ED_P = \frac{p'}{2n} + o(n^{-1}),
\end{equation}
where $p'\triangleq \sum_{\bj(l-1)}(p+1)(p_{\bj(1)}+1)\cdots (p_{\bj(l-1)}+1)-1$, namely 
 the total number of regions $A_{\bj(l)}$ in \eqref{partition_moving} minus one (i.e., the number of free parameters of the multinomial distributions).
\end{theo}
\begin{proof}
Refer to Appendix.
\end{proof}
The $n^{-2}$-order term in $ED_P$ is not as simple as that in $ED_I$ and seems rather difficult to be used in practice. For both methods, $p'$ is equal to the number of free parameters of the corresponding multinomial distribution. Thus, for the $n^{-1}$-order term, there is no difference in the estimation efficiency. The superiority of the moving region method in the $n^{-2}$-order term for a general $k\geq 2$ has not been proven.
%
%
%
%
%
%
%
\section{Two Sample Problem}
\label{Two Sample Problem}
We now return to the two-sample problem and explain the basic approach. To check the closeness of the mother distribution and the model distribution through the discretization method, it is simple and easy to set the common regions (bins) independently of the samples and compare the corresponding two multinomial distributions over these bins. However, considering the superiority of the moving interval method for a one-sample problem when $\bm{X}$ is scalar, it seems better to adopt the moving region method for determining common bins. A question arises as to which sample is to be used ($\xf$ or $\xs$) in \eqref{two_sample} to decide the moving regions. In normal situations, $\xs$ is easily obtained through simulation; hence, $n_2$ could be much larger than $n_1$, and it seems intuitive to construct a common bin based on $\xs$. 

Suppose that we partition $\mathcal{X}$ based on $\xs$ as in Section \ref{discre_vector_X}; then, $A_{\bm{j}(l)}$ are given as in \eqref{partition_moving}.

The discretized mother and model distributions are multinomial distributions over the bins $A_{\bm{j}(l)}$ which are given by the corresponding probabilities
\begin{equation}
\begin{split}
m^{(1)}&\triangleq\{\mf{\bm{j}(l)}\},\quad \mf{\bm{j}(l)} \triangleq P(\bm{X} \in A_{\bm{j}(l)} | \text{the mother distribution}),\\ 
m^{(2)}&\triangleq\{\ms{\bm{j}(l)}\},\quad\ms{\bm{j}(l)} \triangleq P(\bm{X} \in A_{\bm{j}(l)} | \text{the model distribution}).
\end{split}
\end{equation}
The corresponding estimator is given by:
\begin{equation}
\label{def_m_hat}
\begin{aligned}
\mfe{} &\triangleq \{\mfe{\bm{j}(l)}\}, \qquad \mfe{\bm{j}(l)} 
\triangleq \#\{\xf \ |\ \xf \in A_{\bm{j}(l)} \}/n_1, \\
\mse{} &\triangleq \{\mse{\bm{j}(l)}\},\qquad\mse{\bm{j}(l)} \triangleq 1/(p'_{\bm{j}(l-1)}+1),
\end{aligned}
\end{equation}
where $\#$ is the set size function. 

From now on, we focus on the Hellinger distance ($\alpha = 0$) and use it for the comparison of the two distributions for the following reasons. First, it is the only divergence among $\alpha$-divergences that satisfies the property of distance. That is, it satisfies the symmetricity and triangular equation. Second, it avoids the zero-estimation problem. In practice, we often encounter a situation wherein the estimator $\mfe{\bm{j}(l)}$ becomes zero and $\alpha$-divergence diverges. If $-1 < \alpha < 1$, then the $\alpha$- divergence does not diverge. 

Now, the evaluation of the closeness between the mother distribution and the model distribution is given by (we use notation $D[\mf{} : \ms{}]$ instead of $\overset{0}{D}[\mf{} : \ms{}]) $
\begin{equation}
\label{target_estimated}
ED[\mf{} : \ms{}] \triangleq E\Bigl[D[\mf{} : \ms{}]\Bigr] = 2E\Bigl[\sum_{\bm{j}(l)} \Bigl(\sqrt{m^{(1)}_{\bm{j}(l)}}-\sqrt{m^{(2)}_{\bm{j}(l)}}\Bigr)^2\Bigr]
\end{equation}
We estimate $ED[\mf{} : \ms{}]$ as
\begin{equation}
\label{estimator}
D[\mfe{} : \mse{}] = 2\sum_{\bm{j}(l)}  \Bigl(\sqrt{\mfe{\bm{j}(l)}}-\sqrt{\mse{\bm{j}(l)}}\Bigr)^2.
\end{equation}
%
%
%
%
%
The following theorem provides the upper limit of the asymptotic bias of the estimator \eqref{estimator}.
\begin{theo}
\label{bias_two_sample}
The following inequality holds:
\begin{equation}
\label{ineq_bias}
ED[\mf{} : \ms{}] \leq ED[\mfe{} : \mse{}]  + \frac{\sqrt{8p'}}{\sqrt{n_2}} + o(n_1^{-1}) + o(n_2^{-1/2}).
\end{equation}
\end{theo}
\begin{proof}
Notice that
\begin{align*}
&2\sum_{\bm{j}(l)} \Bigl(\sqrt{m^{(1)}_{\bm{j}(l)}}-\sqrt{m^{(2)}_{\bm{j}(l)}}\Bigr)^2\\
&2\sum_{\bm{j}(l)} \Bigl( \sqrt{\mf{\bm{j}(l)}} - \sqrt{\mfe{\bm{j}(l)}} +\sqrt{\mfe{\bm{j}(l)}} - \sqrt{\mse{\bm{j}(l)}} +\sqrt{\mse{\bm{j}(l)}} - \sqrt{\ms{\bm{j}(l)}}\Bigr)^2 \\
& = D\Bigl[ \mf{} : \mfe{} \Bigr] + D\Bigl[\mfe{} : \mse{} \Bigr] + D\Bigl[\mse{} : \ms{} \Bigr] \\
&\qquad + 4 \sum_{\bm{j}(l)} \Bigl(\sqrt{\mf{\bm{j}(l)}} - \sqrt{\mfe{\bm{j}(l)}}\Bigr)\Bigl(\sqrt{\mfe{\bm{j}(l)}} - \sqrt{\mse{\bm{j}(l)}}\Bigr)\\
&\qquad + 4 \sum_{\bm{j}(l)} \Bigl(\sqrt{\mf{\bm{j}(l)}} - \sqrt{\mse{\bm{j}(l)}}\Bigr)\Bigl(\sqrt{\mse{\bm{j}(l)}} - \sqrt{\ms{\bm{j}(l)}}\Bigr).
\end{align*}
Therefore, we have:
\begin{equation}
\label{ED_expression}
\begin{aligned}
&ED[\mf{} : \ms{}]\\
&=ED\Bigl[\mfe{} : \mse{} \Bigr] + ED\Bigl[ \mf{} : \mfe{} \Bigr] + ED\Bigl[\mse{} : \ms{} \Bigr]\\
&\qquad +4E\Bigl[\sum_{\bm{j}(l)} \Bigl(\sqrt{\mfe{\bm{j}(l)}} - \sqrt{\mf{\bm{j}(l)}}\Bigr)   \sqrt{\mse{\bm{j}(l)}} \Bigr] \text{(say A1)}\\
&\qquad +4 E\Bigl[\sum_{\bm{j}(l)} \Bigl( \sqrt{\mf{\bm{j}(l)}}-\sqrt{\mfe{\bm{j}(l)}}\Bigr)\sqrt{\mfe{\bm{j}(l)}}\Bigr] \text{(say A2)}\\
&\qquad + 4E\Bigl[ \sum_{\bm{j}(l)} \Bigl(\sqrt{\mf{\bm{j}(l)}} - \sqrt{\mse{\bm{j}(l)}}\Bigr)\Bigl(\sqrt{\mse{\bm{j}(l)}} - \sqrt{\ms{\bm{j}(l)}}\Bigr)\Bigr] \text{(say A3)}.
\end{aligned}
\end{equation}
First, we evaluate A1. Notably, the conditional asymptotic distribution of $\mfe{\bm{j}(l)}$ when $\xs{}$ is given by:
\[
\sqrt{n_1} (\mfe{\bm{j}(l)} - \mf{\bm{j}(l)} ) \stackrel{d}{\to} N(0, \mf{\bm{j}(l)}(1-\mf{\bm{j}(l)})),
\]
Thus, the Taylor expansion of $\sqrt{\mfe{\bm{j}(l)}}$ around $\sqrt{\mf{\bm{j}(l)}}$ leads to:
\begin{align*}
\sqrt{\mfe{\bm{j}(l)}} - \sqrt{\mf{\bm{j}(l)}} & = \frac{1}{2}\Bigl(\mfe{\bm{j}(l)} - \mf{\bm{j}(l)} \Bigr)\Bigl(\mf{\bm{j}(l)}\Bigr)^{-1/2} \\
& \qquad -\frac{1}{8}\Bigl(\mfe{\bm{j}(l)} - \mf{\bm{j}(l)} \Bigr)^2\Bigl(\mf{\bm{j}(l)}\Bigr)^{-3/2} + o_p(n_1^{-1})\\
\end{align*}
From this expansion, we have 
\[
E_{\xf}[\mfe{\bm{j}(l)}-\mf{\bm{j}(l)} | \xs ]=0,\qquad E_{\xf} [(\mfe{\bm{j}(l)}-\mf{\bm{j}(l)})^2 | \xs]=n_1^{-1}\mf{\bm{j}(l)}(1-\mf{\bm{j}(l)}),
\]
and
\[
E_{\xf}\Bigl[ \sqrt{\mfe{\bm{j}(l)}} - \sqrt{\mf{\bm{j}(l)}}  \Big| \xs \Bigr] = -\frac{1}{8n_1} \Bigl(\mf{\bm{j}(l)}\Bigr)^{-1/2}\Bigl(1-\mf{\bm{j}(l)}\Bigr) + o(n_1^{-1}).
\]
Consequently, A1 was evaluated as follows: 
\begin{equation}
\begin{aligned}
\label{A1}
&4E\Bigl[\sum_{\bm{j}(l)} \Bigl(\sqrt{\mfe{\bm{j}(l)}} - \sqrt{\mf{\bm{j}(l)}}\Bigr)   \sqrt{\mse{\bm{j}(l)}} \Bigr] \\
&= -\frac{1}{2n_1} \sum_{\bm{j}(l)} E_{\xs}\Bigl[ \Bigl(\mf{\bm{j}(l)}\Bigr)^{-1/2}\Bigl(1-\mf{\bm{j}(l)} \Bigr) \Bigr]\sqrt{\mse{\bm{j}(l)}} + o(n_1^{-1}) \\
&\leq  o(n_1^{-1}).
\end{aligned}
\end{equation}
Next, we proceed to calculate A2. Using the relation
\[
4\sum_{\bm{j}(l)} \sqrt{\mf{\bm{j}(l)}}\sqrt{\mfe{\bm{j}(l)}} = -D[\mf{} : \mfe{}] + 4,
\]
we have
\begin{align*}
&4\sum_{\bm{j}(l)} \Bigl( \sqrt{\mf{\bm{j}(l)}}-\sqrt{\mfe{\bm{j}(l)}}\Bigr)\sqrt{\mfe{\bm{j}(l)}}\\
&= 4\sum_{\bm{j}(l)} \sqrt{\mf{\bm{j}(l)}}\sqrt{\mfe{\bm{j}(l)}} -4 \\
&= -D[\mf{} : \mfe{}].
\end{align*}
Therefore, A2 equals
\begin{equation}
\label{A2}
4 E\Bigl[\sum_{\bm{j}(l)} \Bigl( \sqrt{\mf{\bm{j}(l)}}-\sqrt{\mfe{\bm{j}(l)}}\Bigr)\sqrt{\mfe{\bm{j}(l)}}\Bigr] = -ED[\mf{} : \mfe{}].
\end{equation}
For A3, by the Cauchy-Schwarz inequality, we have
\begin{align*}
&4\Bigl | \sum_{\bm{j}(l)} \Bigl( \sqrt{\mf{\bm{j}(l)}}-\sqrt{\mse{\bm{j}(l)}}\Bigr)\Bigl(\sqrt{\mse{\bm{j}(l)}}-\sqrt{\ms{\bm{j}(l)}}\Bigr)\Bigr| \\
&\leq 2 \Bigl(D[\mf{} : \mse{}]\Bigr)^{1/2} \Bigl(D[\mse{} : \ms{}]\Bigr)^{1/2}\\
&\leq  4\Bigl(D[\mse{} : \ms{}]\Bigr)^{1/2} ,
\end{align*}
where the final equality is because
\[
D[\mf{} : \mse{}] = 4- 4\sum_{\bm{j}(l)} \sqrt{\mf{\bm{j}(l)} \mse{\bm{j}(l)}} \leq 4.
\]
Therefore, A3 was evaluated as follows:
\begin{equation}
\label{A3}
\begin{aligned}
&\Bigl | 4E\Bigl[\sum_{\bm{j}(l)}  \Bigl( \sqrt{\mf{\bm{j}(l)}}-\sqrt{\mse{\bm{j}(l)}}\Bigr)\Bigl(\sqrt{\mse{\bm{j}(l)}}-\sqrt{\ms{\bm{j}(l)}}\Bigr)\Bigr]\Bigr| \\
&\leq 4 E\Bigl[\Bigl(D[\mse{} : \ms{}]\Bigr)^{1/2}\Bigr] \\
&\leq 4 \Bigl(ED[\mse{} : \ms{}]\Bigr)^{1/2},
\end{aligned}
\end{equation}
where the last inequality comes from Jensen's inequality. 

From \eqref{A1}, \eqref{A2}, \eqref{A3}, and \eqref{ED_expression}, we have:
\begin{equation}
\begin{aligned}
&ED[\mf{} : \ms{}]\\
&\leq ED\Bigl[\mfe{} : \mse{} \Bigr]  + ED\Bigl[\mse{} : \ms{} \Bigr] +4 \Bigl(ED[\mse{} : \ms{}]\Bigr)^{1/2} + o(n_1^{-1}).\\
\end{aligned}
\end{equation}
From \eqref{ED_P_asymp}, we have:
\[
ED\Bigl[\mse{} : \ms{} \Bigr] = \frac{p'}{2n_2} + o(n_2^{-1}),
\]
thus completing the proof.
\end{proof}
As in \eqref{ineq_bias}, $n_2$ is in the order of 1/2, whereas that of $n_1$ equals 1. Hence, a relatively large $n_2$ is evidently required. For this reason, the sample from the model is used to formulate bins in the moving region method. We are often able to draw the sample from the model at a relatively low cost compared to observations or experiments for the mother distribution.

In relation to the above proof, we illustrate the disadvantage of using the fixed region method for the two-sample problem. For simplicity, we consider the case where $\bm{X}$ is a scalar with support $(-\infty,\ \infty)$. Taking fixed intervals,
\[
I_i^* \triangleq (\xi_i^{(2)}\ \xi_{i+1}^{(2)}), \quad i=0,\ldots,p,
\]
where $\xi_i^{(2)}$ is the percentile of the model distribution, defined as:
\[
\xi_i^{(2)} \triangleq F_{\xs}^{-1}(i/p+1), \quad i= 1,\ldots,p, \qquad \xi^{(2)}_0 = -\infty,\quad \xi^{(2)}_{p+1} = \infty,
\]
where  $F_{\xs}$ is the cumulative distribution function of the model distribution.

The corresponding multinomial distributions over these intervals deduced from the mother and model distributions are given by parameters
\[
m^{(1|2)} \triangleq (m^{(1|2)}_0,\ldots,m^{(1|2)}_p),\qquad m^{(1|2)}_i \triangleq P(\xf \in I_i^*),\ i=0,\ldots,p
\]
and $m=(1/(p+1),\ldots,1/(p+1))$. We  use $ED[m^{(1|2)} :\, m]$ as the discrepancy measure between the mother and model distributions, and estimate it using $D[\mfe{}\,: \,m]$, where 
\[
\mfe{} \triangleq (\mfe{0},\ldots,\mfe{p}),\qquad \mfe{i}\triangleq \#\{\xf \ |\ \xf \in I_i \}/n_1,\ i=0,\ldots,p,
\]
with $I_i$ defined as in \eqref{def_I_i} with $X_2, n_2$ instead of $X_1, n$.
The bias 
\[
ED[\mfe{}\,: \,m] - ED[m^{(1|2)} \, : \, m]  
\]
needs to be evaluated. This bias depends on the value of the probability density function of the mother distribution (refer to Appendix), making the evaluation more sensitive and difficult than the result of Theorem \ref{bias_two_sample}.
%
%
%
%
%
\section{Criterion for Closeness between Mother and Model Distributions}
\label{criteria}
By combining Theorem \ref{bias_two_sample}, Theorem \ref{ED_I_one_sample} and \eqref{simple_threshold}, the following simple criteria can be derived for the closeness between the mother and model distributions:
\\
\\
-- \textbf{Model Fitness Criterion} --
\\
As the threshold for closeness between two distributions, use  $1/2-\epsilon$ for the Bayes error rate, then if 
\begin{equation}
\label{inequality_closeness}
D[\mfe{} : \mse{}]  + \frac{p'}{2n_1} + \sqrt{\frac{8p'}{n_2}} < 8\epsilon^2,
\end{equation}
 the mother and model distributions are sufficiently close.

We note some points regarding this criterion.
\begin{enumerate}
\item From Theorem \ref{bias_two_sample}, as the estimator of $ED[\mf{} : \ms{}]$, 
\[
D[\mfe{} : \mse{}]  + \sqrt{\frac{8p'}{n_2}}
\]
is given after the lower-side bias correction. The first term $D[\mfe{} : \mse{}]$ is a random variable and converges to $D[\mf{} : \mse{}]$. The following inequality holds
\[
\Bigl | D[\mf{} : \mse{}] - D[\mfe{} : \mse{}] \Bigr | \leq D[\mf{} : \mfe{}].
\]
Applying Theorem \ref{ED_I_one_sample} to the distribution of $\xf$ given $\xs$, we have 
\[
ED[\mf{} : \mfe{}] = \frac{p'}{2n_1} + o(n_1^{-1}).
\]
From these, we have
\[
 E\Bigl[\Bigl | D[\mf{} : \mse{}] - D[\mfe{} : \mse{}] \Bigr| \Bigr] \leq \frac{p'}{2n_1} + o(n_1^{-1}).
\]
Using this inequality, we can evaluate the perturbation of $D[\mfe{} : \mse{}]$ around  $D[\mf{} : \mse{}]$ from the sample size $n_1$, which leads to the criteria \eqref{inequality_closeness}.
\item $n_2$ affects the bias of $ D[\mfe{} : \mse{}] $ in the order of $n_2^{-1/2}$, while the expected perturbation of $D[\mfe{} : \mse{}]$ decreases in $n_1^{-1}$ order. This order difference  indicates that we need a relatively larger size of the sample from the model compared to that from the mother distribution. As mentioned in the Introduction section, it is typically easier to get the sample from the model than from the mother distribution because the former is usually obtained through simulation on a computer.
\end{enumerate}
We now state the simulation result as an example of the use of the criterion.
\\
\\
-- \textit{Simulation Results} -- \\
\\
Case1: Multivariate Normal Distributions\\
We conducted a simulation study under the following conditions: the sample $\xf_i,\ i=1,\ldots,n_1$ from the mother distribution is independently and identically distributed as the $k$-dimensional normal distribution, that is, 
\[
\xf_i \sim N(\alpha, I_k + \beta V),
\]
where $I_k$ is the $k$-dimensional identity matrix and $(V)_{ij} = 0.95^{|i-j|}$ for $i, j=1,\ldots,p$. For the model distribution, we used the $k$-dimensional standard normal distribution; $\xs_i,\ i=1,\ldots,n_2$ is independently and identically distributed as
\[
\xs_i \sim N(\bm{0}, I_k).
\]

First, we performed a simulation for the case wherein both the dimension $k$ and the partition depth $l$ are equal to 3, and each variable is partitioned by the sample quantile, that is, $p=p_{\bj(1)}=p_{\bj(2)}=3$ hence, $p'= (3+1)^3-1=63.$ We changed the similarity between the mother and model distributions as $(\alpha, \beta)=(0,0), (0.01,0.01),(0.1,0.1),(1,1)$. We also calculated $D[\mfe{} : \mse{}]$ and the entire value of the left-hand side in \eqref{inequality_closeness}, including the bias correction term
\begin{equation}
\label{bias_correct}
\frac{p'}{2n_1} + \sqrt{\frac{8p'}{n_2}} .
\end{equation}

If we adopt the threshold $\epsilon$ of $0.05$, the right-hand side of \eqref{inequality_closeness} equals $0.02$. Originally, we simulated all possible pairings of $n_1$ and $n_2$ from $10^i,\ i=3,\ldots,7$. As \eqref{bias_correct} is over 0.02, unless $n_2=10^7$ and $n_1\geq 10^4$, we state the result only for these values of $n_1$ and $n_2$. Tables \ref{albeta_0} -- \ref{albeta_1} present the results. 

In the case $(\alpha, \beta)=(0, 0)$(which means that both distributions are the same) and the case $(\alpha, \beta)=(0.01, 0.01)$ from the model fitness criteria \eqref{inequality_closeness} with $\alpha=0.05$, we can conclude that the mother and model distributions are sufficiently close. On the other hand, when $(\alpha, \beta)=(0.1, 0.1)$ or $(\alpha, \beta)=(1, 1)$, both distributions are not judged to be sufficiently close. Notably, if we loosen the threshold to $\alpha=0.1$, then they are sufficiently close when $(\alpha, \beta)=(0.1, 0.1)$.

Second, we perform a simulation for the case $k=10$, $p=p_{\bj(1)}=\cdots =p_{\bj(9)}=3$. If we take the full depth of the partition (i.e., $l=10$), then $p'$ becomes $(3+1)^{10}-1=1048575$ and the sample size required (especially for $n_2$) is prohibitively large. For example, to satisfy the condition that \eqref{bias_correct} is less than 0.02, $n_2$ must be greater than $2*10^{10}$. 

In many practical cases,  two-dimensional distributions of paired variables are of interest in view of the similarity between the two distributions. Hence, we choose $l$ to be $2$ and consider all possible pairings between the two variables. Tables \ref{2dim_n1=10^3} and \ref{2dim_n1=10^4} show the entire values of \eqref{inequality_closeness} when $n_1=10^3,\ n_2=10^7$ (Table 3) and $n_1=10^4,\ n_2=10^7$ (Table 4) for the case $(\alpha, \beta)=(0.1, 0.1)$. For every pairing, we observe that, when $n_1=10^4, n_2=10^7$, the bias is small enough to conclude that the two distributions are similar under the criteria of $\epsilon=0.05$, whereas the bias is too large to reach a conclusion when $n_1=10^3, n_2=10^7$.
\\
\\
Case 2: Bayes Model\\
The posterior distribution in  Bayes model is usually complicated and its density function is not explicitly gained. In most cases, we only have a MCMC sample generated from the posterior distribution. Here, as an example, we consider the following normal regression model using the dataset on a power plant from the UCI machine learning repository \\
(https://archive.ics.uci.edu/ml/datasets/Combined+Cycle+Power+Plant) ;
\[
y =  \beta_1  + \sum_{i=2}^5 \beta_i  x_i +\epsilon,\qquad \epsilon \sim N(0, \sigma)
\]
where $y$ is ``Electrical Output'', and $x_i$'s are respectively ``Temperature'', ``Ambient Pressure'', ``Relative Humidity'' and ``Exhaust Vacuum'' . We consider Bayes modeling with the following prior distributions
\[
\beta_1 \sim \text{Cauchy(0, 10)},\qquad \beta_i \sim \text{Cauchy(0,2.5)},\ i=2,\ldots,5,\qquad \sigma \sim \text{t(5,5,1)},
\]
where Cauchy(0, 10) means the Cauchy distribution with center 0 and scale 10, t(5,5,1) means $t$-distribution with d.f. 5, center 5, and scale 1 (see \cite{Gelman_et.al} for the weakly informative prior distributions).
Let
\[
\hat{y}_j = \beta_1 + \sum_{i=2}^5 \beta_i x_{i,j}, \qquad  j=1,\ldots, N
\]
be the predictor of $y_j$ for each observation $(x_{2,j},\ldots, x_{5,j})$ where $N$ is the sample size. Given the posterior distributions of $\beta$'s,  $\hat{y}$ has some distribution. We can take this Bayes model as the distribution model of $\hat{y}$ that mimics the distribution of the original $y$. However the exact density functions of $\hat{y}$ and $y$ are respectively complicated and unknown, hence we need to measure closeness between $\hat{y}$ and $y$ from their samples. In the simulation, as MCMC posterior sample, we gained 4000 $\beta=(\beta_1,\ldots,\beta_5)$. Since $N=9568$ ($n_1$ in \eqref{inequality_closeness}), we have the sample of $\hat{y}$ of size $4000N=3827\text{e03}$ ($n_2$ in \eqref{inequality_closeness}). 

We set $p'$ to be 10 and drew a histogram for both $\hat{y}$ and $y$ (see Figure \ref{two_hist}) .
\begin{figure}
\caption{Histograms of $\hat{y}$ and $y$}
\centering
\includegraphics[width=15cm]{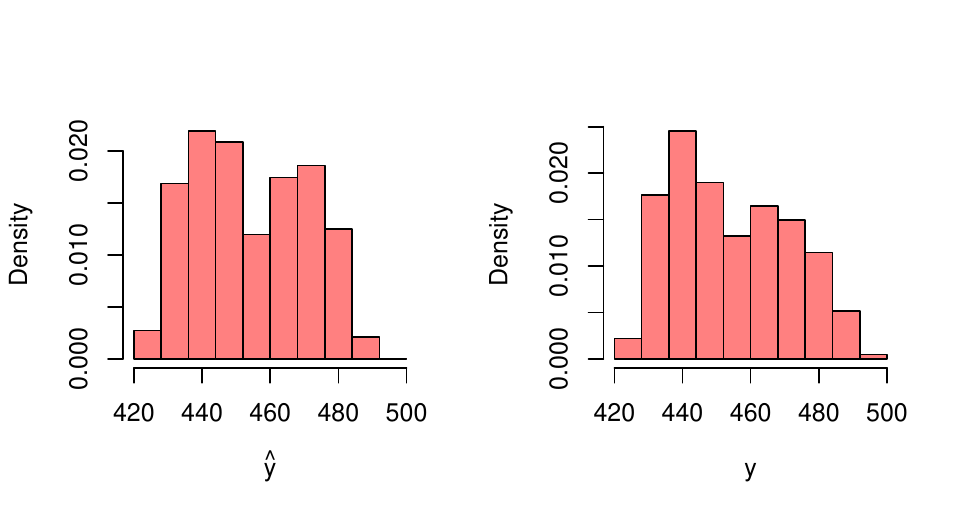}
\label{two_hist}
\end{figure}
$p$-value for two-sided Kolmogolov-Smirnov  statistic is 7.587e-08 and the null hypothesis is rejected even for the extremely low significance level.

We have the following results on \eqref{inequality_closeness};
\[
D[\mfe{} : \mse{}] \doteqdot 0.0113,\qquad \frac{p'}{2n_1} + \sqrt{\frac{8p'}{n_2}}\doteqdot 0.0020
\]
and the whole value of the left side of \eqref{inequality_closeness} is approximately 0.0133. The solution of $\epsilon$ for the equation of 
$8\epsilon^2 = 0.0133$ gives $\epsilon\doteqdot 0.04$, which means the corresponding Bayes error rate $0.46$. Consequently under 0.45 Bayes error rate, the distribution model is sufficiently close to the true distribution.
%
%
%
%
%
%
\section{Appendix}
%
%
%
%

\subsection{Proof of Theorem \ref{Theo:Error and Div}}
	Let $\{S^{(n)}_i \ | \  i=1,\ldots,m(n)\}$ and $(n=1,2,\ldots)$ be a sequence of partitions of $\mathcal{X}$. 
\[
\mathcal{X} =\bigcup_{i=1}^{m(n)} S^{(n)}_i,\qquad S^{(n)}_i \bigcap S^{(n)}_j = \phi \text{ for $i \ne j$}.
\]
We choose $S^{(n)}_i$ and the associated step functions $\tilde{g}^{(n)}_j(x) = \sum_{i=1}^{m(n)} c^{(n)}_{ji} I(x \in S^{(n)}_i),\ j=1,2$, such that 
\begin{align}
\frac{1}{2}\int_{\mathcal{X}} \min\Bigl(\tilde{g}^{(n)}_1(x) , \tilde{g}^{(n)}_2(x)\Bigr) d\mu
 &= \frac{1}{2}\sum_{i=1}^{m(n)}  \int_{S^{(n)}_i} \min(c^{(n)}_{1i}, c^{(n)}_{2i})  d\mu \nonumber\\
 &=\frac{1}{2}\sum_{i=1}^{m(n)}  \min(1, c^{(n)}_{2i}/c^{(n)}_{1i}) \int_{S^{(n)}_i} c^{(n)}_{1i} d\mu \label{appro_error}\\
\int_{\mathcal{X}} \tilde{g}^{(n)}_1(x) f \Bigl( \tilde{g}^{(n)}_2(x)/\tilde{g}^{(n)}_1(x) \Bigr) d\mu &= \sum_{i=1}^{m(n)} f (c^{(n)}_{2i}/c^{(n)}_{1i} ) \int_{S^{(n)}_i} c^{(n)}_{1i} d\mu, \label{appro_dive}
\end{align}  
converges as $n \to \infty$ to
	\begin{align*}
		&Er[g_1(x)\, | \, g_2(x)] = \frac{1}{2}\int_{\mathcal{X}} \min\Bigl(g_1(x), g_2(x)\Bigr) d\mu, \\
		&D_f[g_1(x) \,| \, g_2(x)] = \int_{\mathcal{X}} g_1(x) f \Bigl( g_2(x)/g_1(x) \Bigr) d\mu.
	\end{align*}
Furthermore, we assume that the partitions satisfy 
	\[
	\int_{S^{(n)}_i} c^{(n)}_{1i} d\mu = 1/m(n), \quad i=1,\ldots,m(n).
	\]
Hereafter, we fix $n$ and delete script $(n)$ for notational simplicity. 
Then, \eqref{appro_error} and \eqref{appro_dive} equal 
	\begin{align*}
		&\frac{1}{2m}\sum_{i=1}^m \min (1, \Delta_i) \ (\triangleq t(\Delta_1,\ldots, \Delta_m) ) \\
		&\frac{1}{m}\sum_{i=1}^m f(\Delta_i),
	\end{align*} 
	where $\Delta_i \triangleq c_{2i}/c_{1i},\ i=1,\ldots,m$. Suppose $D_f[g_1(x) \,| \, g_2(x)]< \delta$. Subsequently, for sufficiently large $n$, 
\begin{equation}
		\label{diver_cond}
		\frac{1}{m} \sum_{i=1}^m f(\Delta_i) \leq  \delta.
\end{equation}
	The lower bound of $t(\Delta_1,\ldots,\Delta_m)$ is searched under the condition of \eqref{diver_cond}. Let 
	\begin{equation}
		\label{delta_sum}
		\tilde{m} \triangleq \sum_{i=1}^m \Delta_i,\qquad \tilde{1} \triangleq  \frac{\tilde{m}}{m}.
	\end{equation}
	As partition $S_i,\ i=1,\ldots,m$ becomes finer,
	\[
	\sum_{i=1}^m \int_{S_i}  c_{2i} d\mu =\sum_{i=1}^m \Delta_i/m = \tilde{1} \to \int_{\mathcal{X}} g_2(x) d\mu  =1.
	\]
	
		Without loss of generality, the following can be assumed:   
	\[
	\Delta_1 \geq  \cdots \geq \Delta_s > 1 > \Delta_{s+1} \geq \cdots \geq \Delta_m > 0,\quad \exists s (\geq 1). 
	\]
	Let $u=m-s$ and 
	\[
	\Delta^+ \triangleq  \frac{1}{s} \sum_{i=1}^s \Delta_i,\qquad \Delta^- \triangleq \frac{1}{u} \sum_{i=s+1}^m \Delta_{i}.
	\]
	Note that 
    
	\begin{align}
    &s\Delta^+ + u \Delta^{-} = \tilde{m} \label{sum_delta}\\
	&t(\underbrace{\Delta^+, \cdots, \Delta^+}_s, \underbrace{\Delta^-,\cdots, \Delta^-}_u) = t(\Delta_1,\ldots,\Delta_m) \label{t_delta}
	\end{align}
	and, because of the convexity of $f(\Delta_i)$,
	\begin{equation}
	\frac{1}{m}\{sf(\Delta^+) + u f(\Delta^-)\} \leq \frac{1}{m}\sum_{i=1}^m f(\Delta_i) \leq \delta. \label{main_cond}
	\end{equation}
With $x \triangleq s/m$, \eqref{sum_delta} and \eqref{main_cond} can be rewritten as
\begin{align}
& x \Delta^+ +(1-x) \Delta^- = \tilde{1} \label{sum_delta_2},\\
& x f(\Delta^+) + (1-x) f(\Delta^-) \leq \delta. \label{main_cond2}
\end{align}
From \eqref{sum_delta_2}, we have:
\begin{align*}
		t(\Delta_1,\ldots,\Delta_m)& = \frac{1}{2m}\sum_{i=1}^m \min (1, \Delta_i)\\
		& = \frac{1}{2m}( s + u\Delta^- ) \\
		& =\frac{1}{2}( x + (1-x) \Delta^- ) \\
		&= \frac{1}{2} \bigl(\tilde{1}+x(1-\Delta^+)\bigr) \ \bigl(\triangleq t(x;\Delta^+) \bigr),
\end{align*}
while \eqref{main_cond2} can be rewritten as 
\begin{equation}
\label{main_cond3}
h(x;\Delta^+) \triangleq x f(\Delta^+) +(1-x)  f\Bigl(\frac{\tilde{1}-x\Delta^+}{1-x}\Bigr) \leq \delta.
\end{equation}
As
\[
  0 < \Delta^- = \frac{\tilde{1}-x\Delta^+}{1-x},
\]
we have
\[
x < \frac{\tilde{1}}{\Delta^+} < 1,
\]
where the second inequality follows from the fact that
\[
\tilde{1} \triangleq \frac{\sum_{i=1}^m \Delta_i}{m} < \Delta^+.
\]
Now, we consider the lower bound of function $t(x;\Delta^+)$ over the interval
$(0, \ \tilde{1}/\Delta^+)$ under restriction \eqref{main_cond3}. Notably,
 $h(x;\Delta^+)$ is an increasing function because
\begin{align*}
		\frac{d }{dx} h(x;\Delta^+)  &= f(\Delta^+) -f\Bigl(\frac{\tilde{1}-x\Delta^+}{1-x}\Bigr)+f'\Bigl(\frac{\tilde{1}-x\Delta^+}{1-x}\Bigr)\frac{\tilde{1}-\Delta^+}{1-x} \\
&\geq  -f\Bigl(\frac{\tilde{1}-x\Delta^+}{1-x}\Bigr)+f'\Bigl(\frac{\tilde{1}-x\Delta^+}{1-x}\Bigr)\frac{\tilde{1}-\Delta^+}{1-x} \\
&\geq f'\Bigl(\frac{\tilde{1}-x\Delta^+}{1-x}\Bigr)\Bigl(1-\frac{\tilde{1}-x\Delta^+}{1-x}\Bigr)+f'\Bigl(\frac{\tilde{1}-x\Delta^+}{1-x}\Bigr)\frac{\tilde{1}-\Delta^+}{1-x} 
\quad\text{($f(x)$ is convex)}\\
&= f'\Bigl(\frac{\tilde{1}-x\Delta^+}{1-x}\Bigr)(1-\Delta^+) > 0. \quad\text{( if $x<1$, $f'(x)<0$)}
\end{align*}
Consequently, if 
\begin{equation}
\label{left_cond_h_yes}
h(\tilde{1}/\Delta^+; \Delta^+) \leq \delta,
\end{equation}
then any $x$ in the domain $(0, \ \tilde{1}/\Delta^+)$ satisfies the inequality \eqref{main_cond3}. However, if \eqref{left_cond_h_yes} does not hold true,
\eqref{main_cond3} is equivalent to $x \in (0,\  x^*]$, where $x^*$ satisfies
\begin{equation}
\label{sol_x^*}
h(x^*;\Delta^+) = x^* f(\Delta^+) +(1-x^*)  f\Bigl(\frac{\tilde{1}-x^*\Delta^+}{1-x^*}\Bigr)=\delta.
\end{equation}
First, we consider that Case \eqref{left_cond_h_yes} holds. From 
\[
h(\tilde{1}/\Delta^+; \Delta^+) = \Bigl(\frac{\tilde{1}}{\Delta^+}\Bigr)f(\Delta^+) +
\Bigl(1-\frac{\tilde{1}}{\Delta^+}\Bigr)f(0),
\]
we have
\begin{align*}
\frac{d\quad}{d\Delta^+}h(\tilde{1}/\Delta^+; \Delta^+)
&= -\frac{\tilde{1}}{(\Delta^+)^2}f(\Delta^+) + \frac{\tilde{1}}{\Delta^+}f'(\Delta^+)+ \frac{\tilde{1}}{(\Delta^+)^2}f(0)\\
&= \frac{\tilde{1}}{(\Delta^+)^2}\bigl(f(0) -f(\Delta^+) +f'(\Delta^+) \Delta^+\bigr)\\
&\geq \frac{\tilde{1}}{(\Delta^+)^2}\bigl( -f(\Delta^+) +f'(\Delta^+) \Delta^+\bigr)\qquad\text{(since $f(0) >0$)}\\
&\geq \frac{\tilde{1}}{(\Delta^+)^2}\bigl(-(\Delta^+ -1)f'(\Delta^+) + f'(\Delta^+)\Delta^+\bigr)\qquad\text{(since $f(x)$ is convex)}\\
&= \frac{\tilde{1}}{(\Delta^+)^2}f'(\Delta^+)\\
& > 0. \qquad\text{(since if $x>1$, then $f'(x) > 0$)}
\end{align*}
Therefore, \eqref{left_cond_h_yes} implies that $\Delta^+ < \Delta^+_*(\tilde{1})$, where $\Delta^+_*(\tilde{1})$ is the solution to the equation 
\begin{equation}
\label{sol_Delta^+_*}
h(\tilde{1}/\Delta^+; \Delta^+)=\Bigl(\frac{\tilde{1}}{\Delta^+}\Bigr)f(\Delta^+) +
\Bigl(1-\frac{\tilde{1}}{\Delta^+}\Bigr)f(0)=\delta.
\end{equation}
Consequently, if \eqref{left_cond_h_yes} holds, then:
\begin{equation}
\label{low_bound_for_cond_yes}
t(x ;\Delta^+) \geq  t(\tilde{1}/(\Delta^+) ;\Delta^+) = \frac{\tilde
{1}}{2\Delta^+} > \frac{\tilde{1}}{2\Delta^+_*(\tilde{1})}
\end{equation}

Now, we consider that case \eqref{left_cond_h_yes} does not hold true. In this case, inequality 
\[
t(x ;\Delta^+) \geq t(x^*; \Delta^+) = \frac{1}{2}\bigl(\tilde{1} + x^*(1-\Delta^+)\bigr) \bigl(\triangleq t^*\bigr)
\]
holds. From \eqref{sol_x^*} and the relation: 
\begin{equation}
		\label{del_by_t_x}
		\Delta^+ = (\tilde{1}-2t^*)/x^* +  1,
\end{equation}
we have
	\[
	x^*f\Bigl(\frac{\tilde{1}-2t^*}{x^*} + 1\Bigr)+(1-x^*) f\Bigl(\frac{2t^*-1}{1-x^*}+1\Bigr) = \delta,
	\]
because $\Delta^+ >1$ and $x^* < \tilde{1}/\Delta^+$; 
\[
0 < x^* < 2t^* < \tilde{1}.
\]
Therefore
\begin{equation}
\label{low_bound_for_cond_no}
t(x;\Delta^+) \geq  \inf{\{t^* \,|\, (x^*,t^*) \in \tilde{A}(\delta)\}}.
\end{equation}
where
\[
\tilde{A}(\delta) \triangleq \Bigl\{(x^*,t^*)\, \Big| \, x^* f \Bigl(\frac{\tilde{1}-2t^*}{x^*} + 1\Bigr)+(1-x^*) f\Bigl(\frac{2t^*-1}{1-x^*}+1\Bigr) = \delta,\quad 0 < x^* < 2t^* < \tilde{1}.\Bigr\}.
\]
From \eqref{low_bound_for_cond_yes} and \eqref{low_bound_for_cond_no}, we have
\[
t(x;\Delta^+) \geq \min \Bigl\{ 
\tilde{1}/\bigl(2\Delta^+_*(\tilde{1})\bigr), \inf{\{t^* \,|\, (x^*,t^*) \in \tilde{A}(\delta)\}}
\Bigr\}.
\]	
If we take the limit as $n \to \infty$, we have the result.
%
%
%
%
%
%
\subsection{Proof of Theorem \ref{asymp_risk_moving}}
Let 
\[
R_{\bj(l)} \triangleq \frac{\hat{m}_{\bj(l)}}{m_{\bj(l)}} -1.
\]
As $f(1)=0,\ f'(1)=0,\ f''(1)=1$, we have the following expansion $D_f[m: \hat{m}]$ with respect to $R_{\bj(l)}$ around zero: 
\begin{align}
\label{expan_D_f}
&D_f[m: \hat{m}]\nonumber\\
&=\sum_{ \bj(l) } m_{\bj(l)} f(1+R_{\bj(l)})\nonumber\\
&=\sum_{\bj(l)} m_{\bj(l)} \Bigl(f(1)+f'(1)R_{\bj(l)}+\frac{1}{2}f''(1)R_{\bj(l)}^2+\frac{1}{6}f^{(3)}(1+R^*_{\bj(l)})R_{\bj(l)}^3\Bigr)
\nonumber\\
&=\frac{1}{2}\sum_{\bj(l)} m_{\bj(l)} R_{\bj(l)}^2 +\frac{1}{6}\sum_{\bj(l)}f^{(3)}(1+R^*_{\bj(l)}) m_{\bj(l)} R_{\bj(l)}^3\nonumber\\
\end{align}
where $R^*_{\bj(l)}$ is  a smooth function of $R_{\bj(l)}$; hence, $\hat{m}_{\bj(l)}$ such that 
\begin{equation}
\begin{split}
\label{cond_R^*}
&0 < R^*_{\bj(l)}(\hat{m}_{\bj(l)}) < R_{\bj(l)}(\hat{m}_{\bj(l)}), \text{ if $R_{\bj(l)} (\hat{m}_{\bj(l)})> 0$},\\
&0  > R^*_{\bj(l)} (\hat{m}_{\bj(l)}) > R_{\bj(l)}(\hat{m}_{\bj(l)})\geq -1, \text{ if $R_{\bj(l)} (\hat{m}_{\bj(l)})< 0$}.
\end{split}
\end{equation}
Suppose
\[
\epsilon \triangleq \min_{0 \leq \hat{m}_{\bj(l)} \leq 1} R^*_{\bj(l)} (\hat{m}_{\bj(l)}).
\]
From \eqref{cond_R^*}, we have:
\[
R^*_{\bj(l)} (\hat{m}_{\bj(l)}) >  -1.
\]
Hence, $\epsilon > -1$, which implies that
\[
R^*_{\bj(l)}+1 \geq \epsilon+1 > 0.
\]
Owing to the condition of the boundedness of $f^{(3)}(x)$ on $(\epsilon+1, \infty)$, we have $B_{\bj(l)}\ (>0)$  such that 
\[
|f^{(3)}(1+R^*_{\bj(l)})| \leq B_{\bj(l)}.
\]
Notice that 
\begin{equation}
\label{eval_f^(3)}
\begin{aligned}
&\Bigl |E\Bigl[f^{(3)}(1+R^*_{\bj(l)}) m_{\bj(l)} R_{\bj(l)}^3\Bigr]\Bigr |\\
&\leq E\Bigl[\Bigl |f^{(3)}(1+R^*_{\bj(l)}) m_{\bj(l)} R_{\bj(l)}^3\Bigr|\Bigr]\\
&\leq m_{\bj(l)} B_{\bj(l)} E\Bigl[R_{\bj(l)}^3\Bigr].\\
\end{aligned}
\end{equation}
We decompose $\hat{m}_{\bj(l)}$ as follows.
\begin{align*}
\hat{m}_{\bj(l)} &= P(X_i \in I_{\bj(i)},\ i=1,\ldots,l)\\
&= P(X_1 \in I_{\bj(1)}) P(X_2 \in I_{\bj(2)} | X_1 \in I_{\bj(1)}) \\
&\quad \times P(X_3 \in I_{\bj(3)} | X_1 \in I_{\bj(1)}, X_2 \in I_{\bj(2)})\\
& \hspace{30mm} \vdots \\
&\quad \times P(X_l \in I_{\bj(l)} | X_1 \in I_{\bj(1)}, \ldots, X_{l-1} \in I_{\bj(l-1)}) \\
&= \prod_{i=1}^l P(X_i \in I_{\bj(i)} | X \in A_{\bj(i-1)}).
\end{align*}
Let the cumulative distribution function of $X_{i}$ under the condition $X \in A_{\bj(i-1)}$ be denoted by $F_{i}(x | A_{\bj(i-1)})$, that is,
\[
F_{i}(x | A_{\bj(i-1)}) = P(X_{i} \leq x | X \in A_{\bj(i-1)}).
\]
The random variables are defined as follows. For $s=1,\ldots,p_{\bj(i-1)}$,
\[
U_{\bj(i-1)i} (\tilde{n}_{\bj(i-1),s}) \triangleq F_{i}(X_{\bj(i-1)i} (\tilde{n}_{\bj(i-1),s}) | A_{\bj(i-1)}). 
\]
$U_{\bj(i-1)i} (\tilde{n}_{\bj(i-1),s}) $ has the same distribution as the $\tilde{n}_{\bj(i-1),s}$ th-order statistics (in ascending order) obtained from an independent sample of size $n(A_{\bj(i-1)})$ from the uniform distribution in the interval $(0, 1)$.
If we normalize these statistics as follows:
\[
\Delta_{\bj(i-1)i}(\tilde{n}_{\bj(i-1),s}) \triangleq \sqrt{n(A_{\bj(i-1)})}(U_{\bj(i-1)} (\tilde{n}_{\bj(i-1),s})-s/(p_{\bj(i-1)}+1)),
\]
$\bigl(\Delta_{\bj(i-1)i}(\tilde{n}_{\bj(i-1),s})\bigr)_{s=1,\ldots,p_{\bj(i-1)}}$ converges to a $p_{\bj(i-1)}$-dimensional normal distribution. (refer to \ [5.4.13] of \cite{Lehmann}).   
Now, we have:
\begin{align*}
&P(X_i \in I_{\bj(i)} | X \in A_{\bj(i-1)}) \\
&= F_{i}(X_{\bj(i-1)i} (\tilde{n}_{\bj(i-1),j_i+1}) | A_{\bj(i-1)}) - F_{i}(X_{\bj(i-1)i} (\tilde{n}_{\bj(i-1),j_i}) | A_{\bj(i-1)}) \\
&= U_{\bj(i-1)i} (\tilde{n}_{\bj(i-1),j_i+1}) - U_{\bj(i-1)i} (\tilde{n}_{\bj(i-1),j_i})\\
&= \frac{1}{p_{\bj(i-1)}+1} + \frac{1}{\sqrt{n(A_{\bj(i-1)})}} \bigl(\Delta_{\bj(i-1)i}(\tilde{n}_{\bj(i-1),j_i+1})-\Delta_{\bj(i-1)i}(\tilde{n}_{\bj(i-1),j_i})\bigr)\\
&= \frac{1}{p_{\bj(i-1)}+1} + \frac{\sqrt{p'_{\bj(i-2)}+1}}{\sqrt{n}} \bigl(\Delta_{\bj(i-1)i}(\tilde{n}_{\bj(i-1),j_i+1})-\Delta_{\bj(i-1)i}(\tilde{n}_{\bj(i-1),j_i})\bigr)+o_p(1/\sqrt{n}).
\end{align*}
The last equation holds because
\begin{align*}
n(A_{\bj(i-1)})& = \frac{n}{(p+1)(p_{\bj(1)}+1)(p_{\bj(2)}+1)\cdots (p_{\bj(i-2)}+1)}+ O_p(1)\\
&= \frac{n}{p'_{\bj(i-2)}+1}+ O_p(1),
\end{align*}
where
\[
p'_{\bj(i-2)} = 
\begin{cases} 
(p+1)(p_{\bj(1)}+1)\cdots(p_{\bj(i-2)}+1)-1&\text{ if $i\geq 3$,} \\
p&\text{ if $i=2$,} \\
0&\text{ if $i=1$.} \\
\end{cases}
\]
Therefore, we have
\begin{align*}
&\hat{m}_{\bj(l)}\\
&=\prod_{i=1}^l \Bigl\{ \frac{1}{p_{\bj(i-1)}+1} + \frac{\sqrt{p'_{\bj(i-2)}+1}}{\sqrt{n}} \bigl(\Delta_{\bj(i-1)i}(\tilde{n}_{\bj(i-1),j_i+1})-\Delta_{\bj(i-1)i}(\tilde{n}_{\bj(i-1),j_i})\bigr)\\
&\qquad\qquad+o_p(1/\sqrt{n})\Bigr\}\\
&=\frac{1}{p'_{\bj(l-1)}+1} + \frac{1}{p'_{\bj(l-1)}+1}n^{-1/2} \\
&\quad \times\Bigl\{\sum_{i=1}^l \bigl(\Delta_{\bj(i-1)i}(\tilde{n}_{\bj(i-1),j_i+1})-\Delta_{\bj(i-1)i}(\tilde{n}_{\bj(i-1),j_i})\bigr)\sqrt{p'_{\bj(i-2)}+1}(p_{\bj(i-1)}+1)\Bigr\} \\
&\quad +o_p(1/\sqrt{n}).
\end{align*}

Now, we have
\begin{equation}
\label{expression_R}
\begin{split}
R_{\bj(l)} &= \frac{1}{\sqrt{n}} \Bigl\{\sum_{i=1}^l \bigl(\Delta_{\bj(i-1)i}(\tilde{n}_{\bj(i-1),j_i+1})-\Delta_{\bj(i-1)i}(\tilde{n}_{\bj(i-1),j_i})\bigr)\\
&\qquad\qquad\times\sqrt{p'_{\bj(i-2)}+1}
(p_{\bj(i-1)}+1)\Bigr\} \\
&\qquad+o_p(1/\sqrt{n}).
\end{split}
\end{equation}
Furthermore,
\begin{align*}
& m_{\bj(l)} R^2_{\bj(l)} \\
& = \frac{1}{p'_{\bj(l-1)}+1} \\
&\qquad \times \Bigl\{\sum_{i=1}^l \bigl(\Delta_{\bj(i-1)i}(\tilde{n}_{\bj(i-1),j_i+1})-\Delta_{\bj(i-1)i}(\tilde{n}_{\bj(i-1),j_i})\bigr)^2
(p'_{\bj(i-2)}+1)(p_{\bj(i-1)}+1)^{2} \\
&\qquad \qquad +2\sum_{1 \leq s < t \leq l} \bigl(\Delta_{\bj(s-1)s}(\tilde{n}_{\bj(s-1),j_s+1})-\Delta_{\bj(s-1)s}(\tilde{n}_{\bj(s-1),j_s})\bigr) \\
&\hspace{30mm}\times\bigl(\Delta_{\bj(t-1)t}(\tilde{n}_{\bj(t-1),j_t+1})-\Delta_{\bj(t-1)t}(\tilde{n}_{\bj(t-1),j_t})\bigr)\\
&\hspace{30mm}\times(p'_{\bj(s-2)}+1)^{1/2}(p'_{\bj(t-2)}+1)^{1/2}
(p_{\bj(s-1)}+1)(p_{\bj(t-1)}+1)\Bigr\} \\
&\qquad+o_p(1/n).
\end{align*}

For $1 \leq s < t \leq l$, we have
\begin{align*}
&E\Bigl[\bigl(\Delta_{\bj(s-1)s}(\tilde{n}_{\bj(s-1),j_s+1})-\Delta_{\bj(s-1)s}(\tilde{n}_{\bj(s-1),j_s})\bigr) \\
&\qquad\times\bigl(\Delta_{\bj(t-1)t}(\tilde{n}_{\bj(t-1),j_t+1})-\Delta_{\bj(t-1)t}(\tilde{n}_{\bj(t-1),j_t})\bigr)\Bigr]\\
&= E_{X_1,\ldots,X_{t-1}}\Bigl[ E\Bigl[ \bigl(\Delta_{\bj(s-1)s}(\tilde{n}_{\bj(s-1),j_s+1})-\Delta_{\bj(s-1)s}(\tilde{n}_{\bj(s-1),j_s})\bigr) \\
&\qquad\times\bigl(\Delta_{\bj(t-1)t}(\tilde{n}_{\bj(t-1),j_t+1})-\Delta_{\bj(t-1)t}(\tilde{n}_{\bj(t-1),j_t})\bigr)\,\Big| X_1,\ldots,X_{t-1}\Bigr]\Bigr],\\
\end{align*}
where $E_{X_1,\ldots,X_{t-1}}[ \cdot ]$ and $E[\cdot | X_1,\ldots,X_{t-1}]$ denote the expectation with respect to $(X_{1\tau},\ldots X_{t-1\tau})_{\tau=1,\ldots,n}$ and the conditional expectation when $(X_{1\tau},\ldots X_{t-1\tau})_{\tau=1,\ldots,n}$, respectively.

Notably,
\begin{align*}
&\Bigl[ E\Bigl[ \bigl(\Delta_{\bj(s-1)s}(\tilde{n}_{\bj(s-1),j_s+1})-\Delta_{\bj(s-1)s}(\tilde{n}_{\bj(s-1),j_s})\bigr) \\
&\qquad\times\bigl(\Delta_{\bj(t-1)t}(\tilde{n}_{\bj(t-1),j_t+1})-\Delta_{\bj(t-1)t}(\tilde{n}_{\bj(t-1),j_t})\bigr)\,\Big| X_1,\ldots,X_{t-1}\Bigr]\\
&= \bigl(\Delta_{\bj(s-1)s}(\tilde{n}_{\bj(s-1),j_s+1})-\Delta_{\bj(s-1)s}(\tilde{n}_{\bj(s-1),j_s})\bigr) \\
&\qquad\times E\Bigl[ \bigl(\Delta_{\bj(t-1)t}(\tilde{n}_{\bj(t-1),j_t+1})-\Delta_{\bj(t-1)t}(\tilde{n}_{\bj(t-1),j_t})\bigr)\,\Big| X_1,\ldots,X_{t-1}\Bigr]
\end{align*}

The first- and second-order conditional moments of $U_{\bj(i-1)i}(\tilde{n}_{\bj(i-1),u}),\ i=1,\ldots,l,\ u=1,\ldots,p_{\bj(i-1)}$  are given by:
\begin{align*}
E\bigl[U_{\bj(i-1)i}(\tilde{n}_{\bj(i-1),u})| X_1,\ldots,X_{i-1}\bigr]&= \frac{\tilde{n}_{\bj(i-1),u}}{n(A_{\bj(i-1)})+1}\\
E\bigl[\bigl(U_{\bj(i-1)i}(\tilde{n}_{\bj(i-1),u})\bigr)^2| X_1,\ldots,X_{i-1}\bigr]&=
\frac{\tilde{n}_{\bj(i-1),u}(\tilde{n}_{\bj(i-1),u}+1)}{\bigl(n(A_{\bj(i-1)})+1\bigr)\bigl(n(A_{\bj(i-1)})+2\bigr)}.\\
\end{align*}
(Refer to the Appendix of  \cite{Sheena_3} for more details).  
For  $i=1,\ldots,k,\ u=1,\ldots,p_{\bj(i-1)}$, let 
\[
r_{\bj(i-1),u} = n(A_{\bj(i-1)})u/(p_{\bj(i-1)}+1) - \tilde{n}_{\bj(i-1),u}.
\]
Then, we have
\begin{align*}
&E\Bigl[ \Delta_{\bj(i-1)i}(\tilde{n}_{\bj(i-1),u})\,\Big| X_1,\ldots,X_{i-1}\Bigr]\\
&= \sqrt{n(A_{\bj(i-1)})} \Bigl\{E\Bigl[ U_{\bj(i-1)i}(\tilde{n}_{\bj(i-1),u})\,\Big| X_1,\ldots,X_{i-1}\Bigr]-\frac{u}{p_{\bj(i-1)}+1}\Bigr\}\\
&= \sqrt{n(A_{\bj(i-1)})} \Bigl\{\frac{\tilde{n}_{\bj(i-1),u}}{n(A_{\bj(i-1)})+1}-\frac{u}{p_{\bj(i-1)}+1}\Bigr\}\\
&= \sqrt{n(A_{\bj(i-1)})}\Bigl\{\frac{n(A_{\bj(i-1)})u/(p_{\bj(i-1)}+1) - r_{\bj(i-1),u}}{n(A_{\bj(i-1)})+1}-\frac{u}{p_{\bj(i-1)}+1}\Bigr\}\\
&= \sqrt{n(A_{\bj(i-1)})}\Bigl\{-\frac{1}{n(A_{\bj(i-1)})+1}\Bigl(\frac{u}{p_{\bj(i-1)}+1}+r_{\bj(i-1),u}\Bigr)\Bigr\},
\end{align*}
which is equal to $O_p(n^{-1/2})$ because $n(A_{\bj(i-1)})=O_p(n)$. Accordingly, we find that
\begin{equation}
\label{Expec_cross}
\begin{split}
&E\Bigl[\bigl(\Delta_{\bj(s-1)s}(\tilde{n}_{\bj(s-1),j_s+1})-\Delta_{\bj(s-1)s}(\tilde{n}_{\bj(s-1),j_s})\bigr) \\
&\qquad\times\bigl(\Delta_{\bj(t-1)t}(\tilde{n}_{\bj(t-1),j_t+1})-\Delta_{\bj(t-1)t}(\tilde{n}_{\bj(t-1),j_t})\bigr)\Bigr]=O(n^{-1}).
\end{split}
\end{equation}
We also have
\begin{equation} 
\label{Expec_square}
\begin{split}
&E[\bigl(\Delta_{\bj(i-1)i}(\tilde{n}_{\bj(i-1),j_i+1})-\Delta_{\bj(i-1)i}(\tilde{n}_{\bj(i-1),j_i})\bigr)^2]\\
&=\frac{1}{p_{\bj(i-1)}+1}-\frac{1}{(p_{\bj(i-1)}+1)^2}+O(n^{-1})=\frac{p_{\bj(i-1)}}{(p_{\bj(i-1)}+1)^2}+O(n^{-1})
\end{split}
\end{equation}
for $i=1,\ldots,l$. (Refer to the last paragraph on  P25 of \cite{Sheena_3} for more details.)
From \eqref{Expec_cross} and \eqref{Expec_square}, we have
\begin{align}
E[ m_{\bj(l)} R^2_{\bj(l)}]&=\frac{1}{n(p'_{\bj(l-1)}+1)}\sum_{i=1}^lp_{\bj(i-1)}(p'_{\bj(i-2)}+1)+o(n^{-1})\nonumber\\
&=\frac{1}{n(p'_{\bj(l-1)}+1)}\sum_{i=1}^l \{(p_{\bj(i-1)}+1)(p'_{\bj(i-2)}+1)-(p'_{\bj(i-2)}+1)\}+o(n^{-1})\nonumber\\
&=\frac{1}{n(p'_{\bj(l-1)}+1)}\sum_{i=1}^l \{(p'_{\bj(i-1)}+1)-(p'_{\bj(i-2)}+1)\}+o(n^{-1})\nonumber\\
&=\frac{1}{n(p'_{\bj(l-1)}+1)} \Bigl((p'_{\bj(l-1)}+1)-1\Bigr)+o(n^{-1})\nonumber \\
&=\frac{1}{n}\Bigl(1-\frac{1}{(p'_{\bj(l-1)}+1)}\Bigr)+o(n^{-1}).\nonumber
\end{align}
Furthermore, the following equation holds:
\begin{equation}
\label{eval_E[R^2]}
\frac{1}{2}\sum_{\bj(l)}E[ m_{\bj(l)} R_{\bj(l)}^2] \\
=\frac{1}{2n} \Bigl(p'+1-\sum_{\bj(l)}\frac{1}{(p'_{\bj(l-1)}+1)}\Bigr)+o(n^{-1})=\frac{p'}{2n}+o(n^{-1})
\end{equation}
because
\begin{align*}
\sum_{\bj(l)}\frac{1}{p'_{\bj(l-1)}+1} &= \sum_{\bj(l-1)}\frac{p_{\bj(l-1)}+1}{(p'_{\bj(l-2)}+1)(p_{\bj(l-1)}+1)}\\
& = \sum_{\bj(l-1)}\frac{1}{p'_{\bj(l-2)}+1} \\
&= \sum_{\bj(l-2)}\frac{1}{p'_{\bj(l-3)}+1} \\
&\quad\vdots \\
& =  1.
\end{align*}
Due to \eqref{expression_R}, we have $E[R^3_{\bj(l)}] = o(n^{-1})$. From \eqref{eval_f^(3)}, we find that
\begin{equation}
\label{eval_E[R^3]}
E[f^{(3)}(1+R^*_{\bj(l)}) m_{\bj(l)} R_{\bj(l)}^3]= o(n^{-1}),
\end{equation}
From \eqref{expan_D_f}, \eqref{eval_E[R^2]}, and \eqref{eval_E[R^3]}, \eqref{ED_P_asymp} is obtained.
\subsection{Complement to Theorem 4}
Let 
\[
\mf{} \triangleq (\mf{0},\ldots,\mf{p}),\quad \mf{i}\triangleq P(\xf \in I_i),\ i=0,\ldots,p.
\]
From the relation
\begin{align*}
&D[m^{(1|2)} : m] \\
&= 2\sum_{i=0}^p \Bigl(\sqrt{m_i^{(1|2)}} - \sqrt{\mf{i}} +\sqrt{\mf{i}} -\sqrt{\mfe{i}} +\sqrt{\mfe{i}} - \sqrt{m_i}\Bigr)^2\\
&= D[m^{(1|2)} : \mf{}] + D[\mf{} : \mfe{} ] + D[\mfe{} : m] \\
&\qquad + 4\Bigl(\sqrt{m_i^{(1|2)}} - \sqrt{\mf{i}} \Bigr)\Bigl(\sqrt{\mf{i}}-\sqrt{m_i}\Bigr) \\
&\qquad + 4\Bigl(\sqrt{\mf{i}}-\sqrt{\mfe{i}}\Bigr)\Bigl(\sqrt{\mfe{i}} -\sqrt{m_i}\Bigr),
\end{align*}
we have 
\begin{equation}
\begin{aligned}
&D[m^{(1|2)} : m] - D[\mfe{} : m ] \\
&= D[m^{(1|2)} : \mf{}] + D[\mf{} : \mfe{} ] \\
&\qquad + 4\Bigl(\sqrt{m_i^{(1|2)}} - \sqrt{\mf{i}} \Bigr)\Bigl(\sqrt{\mf{i}}-\sqrt{m_i}\Bigr) \\
&\qquad + 4\Bigl(\sqrt{\mf{i}}-\sqrt{\mfe{i}}\Bigr)\Bigl(\sqrt{\mfe{i}} -\sqrt{m_i}\Bigr).
\end{aligned}
\end{equation}
Note that 
\[
\mf{i} = F_{\xf} (\xs(\tilde{n}_{i+1})) - F_{\xf} (\xs(\tilde{n}_i))\ \text{(see \eqref{def_I_i} for $\tilde{n}_{i}$)},
\]
where $F_{\xf}$ denotes the cumulative distribution function of the mother distribution.
$\mf{i}$ is expanded as a function of $\xs(\tilde{n}_{i+1})$ (or $\xs(\tilde{n}_{i})$) around $\xi^{(2)}_{i+1}$ (or $\xi^{(2)}_{i}$) as follows: 
\begin{align*}
\mf{i}&=F_{\xf} (\xi^{(2)}_{i+1}) - F_{\xf} (\xi^{(2)}_{i}) \\
&\qquad + f_{\xf}(\xi^{(2)}_{i+1})(\xs(\tilde{n}_{i+1})-\xi^{(2)}_{i+1}) -  f_{\xf}(\xi^{(2)}_{i})(\xs(\tilde{n}_{i})-\xi^{(2)}_{i}) + o_p(n_1^{-1/2})\\
&= m^{(1|2)}_i + f_{\xf}(\xi^{(2)}_{i+1})(\xs(\tilde{n}_{i+1})-\xi^{(2)}_{i+1}) -  f_{\xf}(\xi^{(2)}_{i})(\xs(\tilde{n}_{i})-\xi^{(2)}_{i}) + o_p(n_1^{-1/2}).\\
\end{align*}
Therefore, $\mf{i}-m^{(1|2)}_i$ depends on the value of the probability density function of the mother distribution in the first-order asymptotic, similar to $D[m^{(1|2)} : m] - D[\mfe{} : m ]$.

\begin{table}[p]
\caption{$\alpha = \beta = 0,\ n_2=10^7$}
\label{albeta_0}
\centering
\begin{tabular}{c|c|c|c|c}
& $n_1=10^7$ &  $n_1=10^6$ &  $n_1=10^5$ & $n_1=10^4$ \\
\hline
$D[\mfe{} : m]$ & $7.47*10^{-6}$   & $3.47*10^{-5}$  & $3.20*10^{-4}$ &  $3.35*10^{-3}$\\
\hline
L.H.S. of \eqref{inequality_closeness} & $7.11*10^{-3}$  & $7.17*10^{-3} $   & $7.73*10^{-3}$   & $1.36*10^{-3}$   \\
\end{tabular}
\end{table}
\begin{table}[p]
\caption{$\alpha = \beta = 0.01,\ n_2=10^7$}
\label{albeta_001}
\centering
\begin{tabular}{c|c|c|c|c}
& $n_1=10^7$ &  $n_1=10^6$ &  $n_1=10^5$ & $n_1=10^4$ \\
\hline
$D[\mfe{} : m]$ & $2.43*10^{-4}$   & $2.79*10^{-4}$  & $5.53*10^{-4}$ &  $3.49*10^{-3}$\\
\hline
L.H.S. of \eqref{inequality_closeness} & $7.35*10^{-3}$  & $7.41*10^{-3} $   & $7.97*10^{-3}$   & $1.37*10^{-2}$   \\
\end{tabular}
\end{table}
\begin{table}[p]
\caption{$\alpha = \beta = 0.1,\ n_2=10^7$}
\label{albeta_01}
\centering
\begin{tabular}{c|c|c|c|c}
& $n_1=10^7$ &  $n_1=10^6$ &  $n_1=10^5$ & $n_1=10^4$ \\
\hline
$D[\mfe{} : m]$ & $2.05*10^{-2}$   & $2.06*10^{-2}$  & $2.16*10^{-2}$ &  $2.47*10^{-2}$\\
\hline
L.H.S. of \eqref{inequality_closeness} & $2.77*10^{-2}$  & $2.77*10^{-2} $   & $2.90*10^{-2}$   & $3.49*10^{-2}$   \\
\end{tabular}
\end{table}
\begin{table}[p]
\caption{$\alpha = \beta = 1.0,\ n_2=10^7$}
\label{albeta_1}
\centering
\begin{tabular}{c|c|c|c|c}
& $n_1=10^7$ &  $n_1=10^6$ &  $n_1=10^5$ & $n_1=10^4$ \\
\hline
$D[\mfe{} : m]$ & $6.92*10^{-1}$   & $6.91*10^{-1}$  & $6.99*10^{-1}$ &  $6.97*10^{-1}$\\
\hline
L.H.S. of \eqref{inequality_closeness} & $6.99*10^{-1}$   & $6.98*10^{-1}$  & $7.07*10^{-1}$ &  $7.07*10^{-1}$\\
\end{tabular}
\end{table}
\clearpage
\begin{table}
\caption{$\alpha=\beta=0.1, n_1=10^3, n_2=10^7$ , Two-dimensional Marginal Distributions}
\label{2dim_n1=10^3}
    \centering
    \begin{tabular}{|l|l|l|l|l|l|l|l|l|l|}
    \hline
           & 2 & 3 & 4 & 5 & 6 & 7 & 8 & 9 & 10 \\ \hline
        1 & 0.032  & 0.029  & 0.030  & 0.023  & 0.027  & 0.026  & 0.027  & 0.024  & 0.024  \\ \hline
        2 & *  & 0.028  & 0.031  & 0.032  & 0.028  & 0.034  & 0.036  & 0.030  & 0.034  \\ \hline
        3 & *  & *  & 0.029  & 0.031  & 0.029  & 0.034  & 0.037  & 0.031  & 0.032  \\ \hline
        4 & *  & *  & *  & 0.032  & 0.031  & 0.029  & 0.031  & 0.038  & 0.030  \\ \hline
        5 & *  & *  & *  & *  & 0.029  & 0.026  & 0.027  & 0.034  & 0.032  \\ \hline
        6 & *  & *  & *  & *  & *  & 0.029  & 0.029  & 0.032  & 0.028  \\ \hline
        7 & *  & *  & *  & *  & *  & *  & 0.034  & 0.035  & 0.031  \\ \hline
        8 & *  & *  & *  & *  & *  & *  & *  & 0.037  & 0.032  \\ \hline
        9 & *  & *  & *  & *  & *  & *  & *  & *  & 0.032 \\ \hline
    \end{tabular}
\end{table}
\begin{table}
\caption{$\alpha=\beta=0.1, n_1=10^4, n_2=10^7$ Two-dimensional Marginal Distributions}
\label{2dim_n1=10^4}
    \centering
    \begin{tabular}{|l|l|l|l|l|l|l|l|l|l|}
    \hline
          & 2 & 3 & 4 & 5 & 6 & 7 & 8 & 9 & 10 \\ \hline
        1 & 0.018  & 0.017  & 0.017  & 0.016  & 0.016  & 0.016  & 0.015  & 0.016  & 0.017  \\ \hline
        2 & *  & 0.019  & 0.018  & 0.016  & 0.016  & 0.015  & 0.015  & 0.016  & 0.016  \\ \hline
        3 & *  & *  & 0.019  & 0.017  & 0.019  & 0.017  & 0.017  & 0.017  & 0.019  \\ \hline
        4 & *  & *  & *  & 0.018  & 0.017  & 0.018  & 0.017  & 0.018  & 0.019  \\ \hline
        5 & *  & *  & *  & *  & 0.018  & 0.017  & 0.016  & 0.016  & 0.017  \\ \hline
        6 & *  & *  & *  & *  & *  & 0.016  & 0.016  & 0.016  & 0.016  \\ \hline
        7 & *  & *  & *  & *  & *  & *  & 0.015  & 0.016  & 0.017  \\ \hline
        8 & *  & *  & *  & *  & *  & *  & *  & 0.016  & 0.018  \\ \hline
        9 & *  & *  & *  & *  & *  & *  & *  & *  & 0.019 \\ \hline
    \end{tabular}
\end{table}

\end{document}